\documentclass[10pt]{article}
\usepackage[latin1]{inputenc}
\usepackage{amsmath,amsthm,amssymb,amsfonts,amsxtra}
\usepackage{graphicx}

\newtheorem{theorem}{Theorem}[section]
\newtheorem*{theorem*}{Theorem}
\newtheorem{lem}{Lemma}[section]
\newtheorem{corollary}{Corollary}[section]

\newtheorem{proposition}[theorem]{Proposition}

\newcommand{\setR}{\mathbb{R}}
\newcommand{\setN}{\mathbb{N}}
\newcommand{\setC}{\mathbb{C}}
\newcommand{\setZ}{\mathbb{Z}}
\newcommand{\XXint}[3]{{\setbox0=\hbox{$#1{#2#3}{\int}$}
      \vcenter{\hbox{$#2#3$}}\kern-.5\wd0}}

\newcommand{\settmp}[2]{#1\{{#2}#1\}}
\newcommand{\set}[1]{\settmp{}{#1}}

\newcommand{\Bigset}[1]{\settmp{\Big}{#1}}

\newcommand{\normtmp}[2]{#1\lVert{#2}#1\rVert}
\newcommand{\norm}[1]{\normtmp{}{#1}}

\newcommand{\abstmp}[2]{#1\lvert{#2}#1\rvert}
\newcommand{\abs}[1]{\abstmp{}{#1}}

\title{\sc Time-periodic N{\'e}el wall motions\
 }
\author{\sc  Alexander Huber}

\date{}

\begin{document}

\maketitle

\begin{abstract}
In thin ferromagnetic films, the predominance of the magnetic shape anisotropy leads to in-plane magnetizations. The simplest 
domain wall in this geometry is the one-dimensional N{\'e}el wall that connects two magnetizations of opposite sign by a planar $180^\circ$ rotation. In this paper, we perturb the static N{\'e}el wall profile in order to construct time-periodic N{\'e}el wall motions governed by to the Landau-Lifshitz-Gilbert equation. Our construction works within a certain parameter regime and requires the restriction to external magnetic fields with small amplitudes and suitable time averages.
\end{abstract}

\section{Introduction}
The theory of micromagnetism deals with the multiple magnetic phenomena observed in ferromagnetic materials (see for example Brown \cite{brown}, Hubert and Sch\"afer \cite{HS}) like the formation of so-called magnetic domains (regions where the magnetization is almost constant) and the appearance of so-called domain walls (thin transition layers separating the domains). It is a well known fact that the interaction and characteristic of the magnetic structures heavily depend on a variety of parameters (e.g. size and shape of the ferromagnetic sample, material properties, ...). Due to the resulting complexity of the general micromagnetic problem, the mathematical theory of micromagnetism is aimed to provide appropriate approximations for various parameter regimes (see DeSimone, Kohn, M\"uller, and Otto \cite{dkmo} and references therein).

While most of the known mathematical theory has focused on the static case, we studied in \cite{alex_small} (see also \cite{alex_diss}) qualitative properties of time-depending magnetization patterns  modeled by the Landau-Lifshitz-Gilbert equation (LLG).
More precisely, we showed the existence of time-periodic solutions for the full three-dimensional LLG in the regime of soft and small ferromagnetic particles satisfying a certain shape condition. The paper at hand is concerned with time-periodic motions of one-dimensional domain walls which separate two domains with magnetizations of opposite directions by a one-dimensional $180^\text{o}$-transition. These walls are parts of the building blocks for more complicated wall structures and can be classified in the following way:
\begin{itemize}
\item The Bloch wall rotates perpendicular to the 
transition axis. It is observed in bulk materials.
\end{itemize}
\begin{itemize}
\item The N{\'e}el wall rotates in the plane spanned by 
the transition axis and the end states. It is observed in thin films.
\end{itemize}
For one-dimensional domain walls modeled by $m=(m_1,m_2,m_3):\setR \to S^2$, it is possible to derive an energy functional from the full three-dimensional micromagnetic energy by means of dimension reduction. Such a reduction was carried out by Aharoni \cite{aharoni1}, \cite{aharoni2} (see also Garc{\'{\i}}a-Cervera \cite{Garcia} and Melcher \cite{Melcher}), and the resulting energy functional for a ferromagnetic layer of thickness $\delta>0$  reads as follows:
\begin{align*}
  E(m) = d^2 \int_\setR \abs{m'}^2 \, dx + Q \int_\setR (m_1^2 + m_3^2) \, dx + \int_\setR \mathcal{S}_\delta [m_1] m_1 \, dx + \int_\setR ( m_3^2 - \mathcal{S}_\delta [m_3] m_3 ) \, dx
 \, .
\end{align*}
On the right hand side, the first two terms are called exchange energy and anisotropy energy, respectively. 
The exchange energy explains the tendency towards parallel alignment, where the positive material constant \nolinebreak $d$ is called exchange length or Bloch line width. Here we assume that the material has a single easy axis along $e_2=(0,1,0)$ leading to the anisotropy energy with positive quality factor $Q$. Moreover, the magnetization $m$ induces a magnetic field -- the so-called stray field -- causing an energy contribution given by the last two terms on the right hand side, where $ \mathcal{S}_\delta$ is the reduced stray field operator, a Fourier multiplication operator defined  by
\begin{align*}
 f \mapsto \mathcal{S}_\delta[f] = \mathcal{F}^{-1}\big( \sigma_\delta \, \widehat{f}\,\big)
\end{align*}
with real-valued, nonnegative, and bounded symbol
\begin{align*}
 \sigma_\delta (\xi) = 1 - \frac{1- e^{- \delta \abs{\xi}}}{\delta \abs{\xi}} \,.
\end{align*}
 The assumption that $m$ represents a domain wall, that is $m$ connects end states of opposite directions, is reflected by the additional requirement
\begin{align*}
 \lim_{x \to \pm \infty} m(x) = \pm e_2 \,.
\end{align*}
In the model introduced above, the Bloch and N{\'e}el walls correspond to minimizers of $E$ in the cases $m_1\equiv0$ and $m_3\equiv0$, respectively. In an infinitely extended layer, that is $\delta = \infty$, the Bloch wall path completely avoids the occurrence of magnetic volume charges, and as a result, the stray field energy is equal to zero and therefore minimal. For this reason, the energy functional reduces to two terms, and it is possible to compute the Bloch wall profile explicitly. Indeed, with the help of a scaling argument one can restrict oneself to the computation of a single reference profile, and it turns out that the Bloch wall exhibits an exponential decay beyond a transition zone of order unity (see \cite{HS}). 
The  N{\'e}el wall is mathematically more delicate due to the nonvanishing 
contribution of the stray field energy and the resulting dependence on multiple scales. In soft and thin films, the characteristic properties of the N{\'e}el wall are the very long logarithmic tail of the transition profile and logarithmic energy scaling (see \cite{dkmo}, \cite{Garcia}, \cite{Melcher}, and \cite{Melcher2}).

For one-dimensional domain walls $m=m(t,x): \setR \times \setR \to S^2$, LLG with respect to a time-dependent external magnetic field $h_\text{ext}$ is given by
\begin{align*}
 m_t = \alpha \, m \times H_{\text{eff}} - m \times ( m \times H_{\text{eff}}) \, , \quad 
\abs{m}=1 \,, \quad \lim_{x \to \pm \infty} m(\cdot,x) = \pm e_2 \,,
\end{align*}
where
\begin{align*}
 H_{\text{eff}} = d^2 m'' - Q
 (m_1,0,m_3)^T
- (\mathcal{S}_\delta [m_1],0,m_3 - \mathcal{S}_\delta [m_3])^T 
+(0,h_\text{ext},0)^T
\end{align*}
is the effective magnetic field and \textquotedblleft$\times$\textquotedblright{} denotes the usual cross product in $\setR^3$ (we assume that the external magnetic field is parallel to $e_2$). The so-called ``gyromagnetic'' term $\alpha \, m \times H_{\text{eff}}$ describes 
a precession around $H_{\text{eff}}$, whereas the ``damping'' term $- m  \times \big( m \times H_{\text{eff}}\big) $ tries to 
align $m$ with $H_{\text{eff}}$.

The assumption of in-plane magnetizations in the thin film geometry  is incompatible with the dynamics as described by LLG since  gyromagnetic precession generates an energetically unfavorable out-of-plane component.
Certain reduced models for the in-plane components have been considered in \cite{GarciaE},
\cite{KohnSlastikov}, and \cite{CMO} as the film thickness goes to zero. In  \cite{GarciaE} and \cite{KohnSlastikov}, the so-called Gilbert damping factor is held fixed of order one leading to an overdamped limit of LLG for the in-plane components. In \cite{CMO} a different parameter regime is treated. There, the Gilbert damping factor is 
assumed to be comparable to a certain relative thickness, and the resulting  LLG is a damped geometric 
wave equation for the in-plane magnetization components. 

In the presence of a constant external magnetic field, one expects the N{\'e}el wall to move towards the less preferred end state. This is true for the reduced LLG derived in \cite{CMO}, where the authors perturb the static N{\'e}el wall profile in order to construct traveling wave solutions. Their proof relies mainly on the spectral properties of the linearized problem and the implicit function theorem.

The aim of this paper is the construction of time-periodic solutions for LLG in the thin film geometry when the external magnetic field is time-periodic. For this we investigate the full LLG and allow an out-of-plane component for the magnetization, hence no further reduction is done. We assume as in \cite{CMO} that the  material is soft ($Q \ll 1$) and that $\kappa = \delta^{-2} d^2 Q$ is bounded from below. The main difference compared to our work in \cite{alex_small} is that the linearization of the Euler-Lagrange equation possesses a nontrivial kernel. This stems from the translation invariance of the energy functional and requires (compared to \cite{alex_small}) the introduction of an additional parameter, which can be interpreted as time average of the external magnetic field. Starting from the static N{\'e}el wall profile, we then use the continuation method to construct time-periodic solutions in the case of external magnetic fields with small amplitudes \nolinebreak $\lambda$ and certain time averages $\gamma(\lambda)$ (see Theorem \ref{thm:3}). 

The paper is organized as follows: In Section \ref{sec:Neel Wall} we recall properties of static N{\'e}el walls and rewrite LLG in a
suitable coordinate system. The linearization $\mathcal{L}^\epsilon_0$ of LLG is composed of two linear operators, which we analyze separately in Sections \ref{sec:NL1} and \ref{sec:NL2}. The thereby obtained results are used in Section \ref{sec:NL0} to study the spectral properties of the analytic semigroup generated by $\mathcal{L}^\epsilon_0$. These properties are the crucial ingredients for our perturbation argument in the final section.

\paragraph{Notation.} Before we start, we introduce and recall some short hand notations and definitions. We write $L^p = L^p(\setR)$ and $H^k=H^k(\setR)$
for the Lebesgue and Sobolev spaces on $\setR$,
and ``$\int$'' means integration over the whole real line with respect to the one-dimensional Lebesgue measure. We write $u \perp v$ to indicate that the $L^2$-functions $u$ and $v$ are perpendicular in $L^2$, that is  $(u,v)_{L^2} = 0$.
Furthermore, we use the abbreviations
\begin{align*}
 \cdot ' = \frac{d}{dx} \qquad \text{and} \qquad \cdot'' = \Delta = \frac{d^2}{dx^2}
\end{align*}
for the first and second derivatives with respect to a one-dimensional space variable $x$. 
For the Fourier transform, we make the convention
\begin{align*}
 \mathcal{F} u (\xi) = \widehat{u}(\xi) = \frac{1}{\sqrt{2 \pi}} \int 
e^{-\dot{\imath} x \xi } u(x) \, dx
\end{align*}
for functions $u:\setR \to \setC$ whenever this is well-defined. In particular, we have that $(\hat{u},\hat{v})_{L^2} = (u,v)_{L^2}$ for all $u,v \in L^2$. For a Banach space  $X$ and a given $0<\beta<1$. we denote by $C^{0,\beta}_\beta(]0,T],X)$ the set of all bounded functions $f:\,]0,T] \to X$ such that
\begin{align*}
 [f]_{C^{0,\beta}_\beta(]0,T],X)} = \sup_{0<\epsilon<T} \epsilon^\beta [f]_{C^{0,\beta}([\epsilon,T],X)}
 < \infty \,.
\end{align*}
This forms a Banach space with norm defined by
\begin{align*}
\norm{f}_{C^{0,\beta}_\beta(]0,T],X)} = \norm{f}_{C(]0,T],X)} + [f]_{C^{0,\beta}_\beta(]0,T],X)} \,.
\end{align*}
Moreover, we write $C^{1,\beta}_\beta (]0,T],X)$ for the set of all bounded and differentiable functions $f :\,]0,T] \to X$ with derivative $f'$ belonging to $C^{0,\beta}_\beta(]0,T],X)$. Again, this is a Banach space with norm defined by
\begin{align*}
\norm{f}_{C^{1,\beta}_\beta (]0,T],X)} = \norm{f}_{C(]0,T],X)} + \norm{f'}_{C^{0,\beta}_\beta(]0,T],X)}\,.
\end{align*}
Given a Banach space $Y$ and a linear operator $A:D(A) \subset Y \to Y$, we say that $A$ is sectorial if there are constants $\omega \in \setR$, $\theta \in (\pi/2,\pi)$, and $M>0$ such that the resolvent set $\rho(A)$ contains the sector $S_{\theta,\omega} = \set{ \lambda \in \setC \, | \, \lambda \not= \omega, \, \abs{\arg(\lambda - \omega)} < \theta}$,
and the resolvent estimate
\begin{align*}
 \norm{\text{R}(\lambda,A)} \le \frac{M}{\abs{\lambda - \omega}}
\end{align*}
is satisfied for all $\lambda \in  S_{\theta,\omega}$. Moreover, we denote by $(e^{tA})_{t \ge 0}$ the analytic semigroup generated by the sectorial operator $A$ (see for example the book by Lunardi \cite{lunardi} for a self-contained presentation of the theory of sectorial operators and analytic semigroups).

\section{N{\'e}el walls and LLG for N{\'e}el walls}
\label{sec:Neel Wall}
\paragraph{Parameter reduction and N{\'e}el walls.}
We consider the micromagnetic energy functional for static and planar magnetizations $m = (m_1,m_2):\setR \to S^1$. If we rescale space by
\begin{align*}
 x \mapsto \frac{\delta}{Q} x \, .
\end{align*}
then we obtain -- after a further renormalization of the energy by $\delta$ -- the rescaled energy functional
\begin{align*}
 {E_\text{res}^\epsilon}(m) =  \kappa \int \abs{m'}^2 + \int m_1^2 + \frac{1}{\epsilon} \int
 \mathcal{S}_\epsilon [m_1] m_1
\end{align*}
with parameters $\kappa = \delta^{-2}d^2 Q$ and $\epsilon = Q$. 
Moreover, we have the following lemma for $ \mathcal{S}_\epsilon$:
\begin{lem}
\label{lem:reduced stray field}
 The linear operator $\mathcal{S}_\epsilon$ defines a linear and bounded mapping from $H^k$ to $H^k$ for every $k \in \setN \cup \set{0}$. Furthermore, $\mathcal{S}_\epsilon$ satisfies
\begin{align*}
  (\mathcal{S}_\epsilon[u],v)_{L^2} = (u,\mathcal{S}_\epsilon[v])_{L^2}
  \qquad \text{and} \qquad (\mathcal{S}_\epsilon[u], u)_{L^2} \ge 0
\end{align*}
for every $u,v \in L^2$. If $u \in L^2$ is a real-valued function, then $\mathcal{S}_\epsilon[u]$ is real-valued as well.
\end{lem}
\begin{proof}
The boundedness of $\sigma_\epsilon$ combined with the Fourier characterization of the Sobolev spaces implies that $\mathcal{S}_\epsilon:H^k \to H^k$ is well-defined and bounded. Since $\sigma_\epsilon$ is real-valued, we see that 
$\mathcal{S}_\epsilon$ is symmetric on $L^2$, and because of $\sigma_\epsilon \ge 0$, we obtain that $\mathcal{S}_\epsilon$ is positive semidefinite on $L^2$. The remaining statement follows from considering $\sigma_\epsilon(\xi) =
\sigma_\epsilon(-\xi)$ for every $\xi \in \setR$. The lemma is proved.
\end{proof}
As already announced, we assume in the sequel that $\kappa>0$ is fixed (or bounded from below) and vary the (small) parameter $\epsilon>0$. In particular, we are in the regime of soft thin films, where $\frac{\delta}{d} \sim \sqrt{Q}$. It can easily be seen that ${E_\text{res}^\epsilon}$ admits a minimizer $m_\epsilon$ in the set of admissible functions defined by
\begin{align*}
 \Bigset{m: \setR \to S^1 \, | \, m' \in L^2,\, m_1 \in L^2, \,
 m_1(0) = 1,\, \lim_{x \to \pm \infty} m_2(x) = \pm 1} \,.
\end{align*}
In particular, minimizers are centered in the sense that $m^\epsilon_1(0)=1$ and carry out a $180^\circ$ in-plane rotation between the end states $m_\epsilon(-\infty) = (0,-1)$, $m_\epsilon(\infty) = (0,1)$. We call them rescaled N{\'e}el walls. Moreover, minimizers are weak solutions of the Euler-Lagrange equation
\begin{align*}
 -\kappa \, m'' + \frac{1}{\epsilon}
 (\mathcal{S}_\epsilon [m_1],0)^T
 + 
  (m_1,0)^T
 = \Big( \kappa \abs{m'}^2 + \abs{m_1}^2 + \frac{1}{\epsilon} m_1 \mathcal{S}_\epsilon [m_1] \Big)
  m \,.
\end{align*}
From here it follows that rescaled N{\'e}el walls $m_\epsilon$ are smooth, and $m^\epsilon_1$, $\frac{d}{dx}m^\epsilon_2$ belong to $H^k$ for all $k \in \setN$. See for example \cite{Melcher} for the derivation of the Euler-Lagrange equation and a proof of the regularity statement. Moreover, using the notion of rearrangement, it can be seen that rescaled N{\'e}el wall profiles $m^\epsilon_1$ are nonnegative and symmetrically decreasing (see \cite{Melcher}).
\paragraph{The phase function of a rescaled N{\'e}el wall.}
For the purpose of spectral analysis, we introduce for rescaled N{\'e}el walls $m_\epsilon$ as in \cite{CMO} the smooth phase function $\theta_\epsilon:\setR \to \setR$ such that $\theta_\epsilon(0)=0$ and
\begin{align*}
 m_\epsilon = (m^\epsilon_1,m^\epsilon_2) = (\cos \theta_\epsilon, \sin \theta_\epsilon) \,.
\end{align*}
 The rescaled energy functional ${E_\text{res}^\epsilon}$ and the Euler-Lagrange equation in these coordinates read as
\begin{align*}
{E_\text{res}^\epsilon} (\theta) = \kappa \int \abs{\theta'}^2 + \int \cos^2 \theta + \frac{1}{\epsilon} \int \mathcal{S}_\epsilon [\cos \theta] \cos \theta
\end{align*}
and
\begin{align}
 \label{eq:Euler-Lagrange-N}
 \tag*{$(EL)$}
 \kappa \theta'' + \frac{1}{2} \sin (2\theta) + \frac{1}{\epsilon} \mathcal{S}_\epsilon [\cos \theta] \sin \theta =0 \, ,
\end{align}
respectively. From the above stated regularity properties for rescaled N{\'e}el walls $m_\epsilon = (\cos \theta_\epsilon , \sin \theta_\epsilon)$ (or directly from the Euler-Lagrange equation for the phase function), we obtain $\theta'_\epsilon \in H^k$ for all $k \in \setN$. Since $m^\epsilon_1 = \cos \theta_\epsilon$ is nonnegative and $\theta_\epsilon(0)=0$, we find
$- \frac{\pi}{2} \le \theta_\epsilon \le \frac{\pi}{2}$.
We now see that
\begin{align*}
 \lim_{x \to \pm \infty} \theta_\epsilon(x) = \pm \frac{\pi}{2} \, ,
\end{align*}
and since $m^\epsilon_1 = \cos \theta_\epsilon$ is symmetrically decreasing, we obtain that $\theta_\epsilon$ is nondecreasing. As in \cite{CMO} we show that $\theta'_\epsilon(0)>0$:
Set $b(x) = \mathcal{S}_\epsilon [\cos \theta_\epsilon](x)$ and observe that $b$ is continuous and bounded. Now assume that $\theta'_\epsilon(0)=0$. Then $\theta_\epsilon$ solves the ODE
\begin{align*}
 \kappa \theta'' + \frac{1}{2} \sin (2\theta) + \frac{1}{\epsilon} b(x) \sin \theta =0, \quad \theta(0)=0, \quad
 \theta'(0)=0,
\end{align*}
and the uniqueness theorem implies $\theta_\epsilon\equiv0$, a contradiction. 

We summarize the properties of the phase function $\theta_\epsilon$ of a rescaled N{\'e}el wall in the next lemma.
\begin{lem}
\label{lem:Nphasefunction}
 The phase function $\theta_\epsilon$ of a rescaled N{\'e}el wall is smooth and satisfies the following
 properties:
 \begin{enumerate}
  \item[(i)] $\theta_\epsilon(0)=0$, $\theta'_\epsilon \ge 0$, and $\theta'_\epsilon(0)>0$.
  \item[(ii)] $\theta'_\epsilon$, $\cos \theta_\epsilon \in H^k$ for all $k \in \setN$.
  \item[(iii)] $-\frac{\pi}{2} \le \theta_\epsilon \le \frac{\pi}{2}$ and $\lim_{x \to \pm \infty} \theta_\epsilon(x) = \pm \frac{\pi}{2}$.
 \end{enumerate}
\end{lem}
\paragraph{LLG for N{\'e}el walls and choice of coordinates.} Now we consider LLG for one-dimensional domain walls $m=m(t,x):\setR\times\setR \to S^2$.
If we rescale space and time by
\begin{align*}
 x \mapsto \frac{\delta}{Q} x \qquad \text{and} \qquad t \mapsto \frac{1}{Q} t
 \, ,
\end{align*}
respectively, then we obtain the rescaled LLG given by
\begin{align*}
m_t = \alpha \, m \times H_{\text{eff}}^{\text{res}} - m \times ( m \times
 H_{\text{eff}}^{\text{res}} ) \, , \,\, \abs{m}=1 \, , \,\, \lim_{x \to \pm \infty} m(\cdot,x) = \pm e_2 \,,
\end{align*}
with rescaled effective field
\begin{align*}
 H_{\text{eff}}^{\text{res}} = \kappa  m'' -
 (m_1,0,m_3)^T
- \frac{1}{\epsilon}
 (\mathcal{S}_\epsilon [m_1],0, m_3 - \mathcal{S}_\epsilon [m_3])^T 
+ \frac{1}{\epsilon}
 (0, h_\text{ext},0)^T\, .
\end{align*}
If the external magnetic field $h_\text{ext}$ is zero, then the rescaled N{\'e}el wall is a stationary solution for the rescaled LLG with $m_3 \equiv 0$. For $h_\text{ext} \not= 0$ and $\alpha \not= 0$, the assumption $m_3 \equiv 0$ becomes incompatible with LLG since the precession term leads to an out-of-plane component $m_3 \not= 0$. 
In view of the saturation constraint $\abs{m}=1$, we introduce a spherical coordinate system and write
\begin{align*}
 m_1 = \cos \varphi \cos \theta, \quad
 m_2 = \cos \varphi \sin \theta, \quad
 m_3 = \sin \varphi,
\end{align*}
with angles $\varphi,\theta$. In particular, we have $\varphi=0$ and $\theta = \theta_\epsilon$ for the rescaled N{\'e}el wall with phase function \nolinebreak $\theta_\epsilon$. In order to rewrite the rescaled LLG in spherical coordinates, we introduce the matrix
\begin{align*}
 M(m) = 
\begin{pmatrix}
 - \sin \varphi \cos \theta & - \sin \varphi \sin \theta  & \cos \varphi
\\
-\sec \varphi \sin \theta & \hphantom{-} \sec \varphi \cos \theta & 0
\\
\hphantom{-}\cos \varphi \cos \theta & \hphantom{-}\cos \varphi \sin \theta & \sin \varphi
\end{pmatrix}
\end{align*}
with determinant given by $\det M(m) = - \sec \varphi$. For $m$ close to the rescaled N{\'e}el wall, we have $\varphi \approx 0$ and the matrix $M(m)$ becomes invertible. Multiplication of the rescaled LLG with $M(m)$ leads to an equivalent equation in terms of $\varphi,\theta$. A rather long but straightforward calculation shows:
\begin{align*}
\tag*{$(LLG)_\epsilon$}
 \begin{array}{ll}
  \varphi_t = R^\epsilon_1(t,\varphi,\theta,h_\text{ext})\, , & \quad \lim_{x \to \pm \infty} \varphi(\cdot,x) = 0 \, ,
  \\
  \theta_t = R^\epsilon_2(t,\varphi,\theta,h_\text{ext})\, , & \quad \lim_{x \to \pm \infty} \theta(\cdot,x) = \pm \frac{\pi}{2} \, ,
 \end{array}
\end{align*}
where
\begin{align*}
 &R^\epsilon_1(t,\varphi,\theta,h_\text{ext}) 
\\
=& \frac{\alpha}{\epsilon} h_\text{ext} \cos \theta + \frac{1}{\epsilon} \mathcal{S}_\epsilon [\sin \varphi] \cos \varphi
 + \frac{1}{\epsilon} \mathcal{S}_\epsilon [\cos \varphi \cos \theta] \sin \varphi \cos \theta
 + \frac{\alpha}{\epsilon} 
 \mathcal{S}_\epsilon [\cos \varphi \cos \theta] \sin \theta - \frac{1}{\epsilon} h_\text{ext} \sin \varphi \sin \theta 
\\
&- 2 \alpha 
\kappa \sin \varphi \, \varphi' \theta' 
+ \frac{1}{4 \epsilon} \sin (2\varphi) (-2 -\epsilon + \epsilon \cos (2\theta) + 2 \epsilon \kappa (\theta')^2) + \kappa \varphi'' 
+ \frac{\alpha}{2} \cos \varphi \sin (2\theta) 
+ \alpha \kappa \cos \varphi \, \theta''
\\
\intertext{and}
&R^\epsilon_2(t,\varphi,\theta,h_\text{ext}) 
\\
=& -\frac{\alpha}{\epsilon} \mathcal{S}_\epsilon [\sin \varphi] + \frac{1}{\epsilon} h_\text{ext} \cos \theta \sec \varphi + \frac{\alpha}{2 \epsilon} \sin \varphi \big(2+ \epsilon - \epsilon \cos (2\theta) \big) 
+ \frac{1}{2} \sin (2\theta) 
+ \frac{\alpha}{\epsilon} h_\text{ext} \tan \varphi \sin \theta  
\\
&+\frac{1}{\epsilon} \mathcal{S}_\epsilon [ \cos \varphi \cos \theta ] \sec \varphi \sin \theta 
- \frac{\alpha}{\epsilon} \mathcal{S}_\epsilon [ \cos \varphi \cos \theta ] \tan \varphi \cos \theta 
- 2 \kappa \tan \varphi \, \varphi' \theta' - \alpha \kappa \sin \varphi (\theta')^2 
\\
&- \alpha \kappa \sec \varphi \, \varphi'' + \kappa \theta'' \, .
\end{align*}
In the following we investigate $(LLG)_\epsilon$ and construct time-periodic solutions close to the rescaled N{\'e}el wall for time-periodic external magnetic fields $h_\text{ext}$.
\paragraph{Linearization of the rescaled LLG.}
As in \cite{alex_small}, the linearization of $(LLG)_\epsilon$ at the stationary solution is of crucial importance for our arguments. If we set $h_\text{ext}=0$, then the linearization of the right hand side with respect to $(\varphi,\theta)$ at $(\varphi,\theta)=(0,\theta_\epsilon)$ is given by
\begin{align*}
 \mathcal{L}_0^\epsilon = 
\begin{pmatrix}
 \hphantom{-\alpha} \mathcal{L}_1^\epsilon & \alpha \mathcal{L}_2^\epsilon
\\
-\alpha \mathcal{L}_1^\epsilon & \hphantom{\alpha}\mathcal{L}_2^\epsilon
\end{pmatrix}
: H^2 \times H^2 \subset L^2 \times L^2 \to L^2 \times L^2\, ,
\end{align*}
where
\begin{align*}
 \mathcal{L}_1^\epsilon u = \hspace{-0.1cm}\kappa u'' - \frac{1}{\epsilon} u - \frac{1}{2} u 
+ \frac{1}{2} \cos(2 \theta_\epsilon) u + \hspace{-0.1cm}\kappa (\theta'_\epsilon)^2 u + \frac{1}{\epsilon} \mathcal{S}_\epsilon [\cos \theta_\epsilon ] \cos \theta_\epsilon \, u + \frac{1}{\epsilon} \mathcal{S}_\epsilon [u]
\end{align*}
and
\begin{align*}
 \mathcal{L}_2^\epsilon v = \hspace{-0.1cm}\kappa v'' + \cos(2\theta_\epsilon) v - \frac{1}{\epsilon} \mathcal{S}_\epsilon [\sin \theta_\epsilon \, v] \sin \theta_\epsilon + \frac{1}{\epsilon} \mathcal{S}_\epsilon [\cos \theta_\epsilon ] \cos \theta_\epsilon \, v
\end{align*}
for $u,v \in H^2$. We remark that $\mathcal{L}_2^\epsilon$ is the linearization of the Euler-Lagrange equation \ref{eq:Euler-Lagrange-N} for the phase function at $\theta_\epsilon$. In the following two sections, we collect properties of $\mathcal{L}_1^\epsilon$ and $\mathcal{L}_2^\epsilon$ in order to analyze the spectrum of $\mathcal{L}_0^\epsilon$ in Section \ref{sec:NL0}. To be more precise, we show that $0$ is an isolated point in $\sigma(\mathcal{L}_0^\epsilon)$ with one-dimensional eigenspace spanned by $(0,\theta'_\epsilon)$ and $\sigma(\mathcal{L}_0^\epsilon) \cap \dot{\imath}\setR  = \set{0}$, provided the parameter $\epsilon>0$ is small enough.
\section{The linear operator $\mathcal{L}_1^\epsilon$}
\label{sec:NL1}
In this section we prove that $\mathcal{L}_1^\epsilon$ is self-adjoint and invertible for
$\epsilon$ small enough. For this we need a priori estimates for the phase function $\theta_\epsilon$ of a rescaled N{\'e}el wall independent of the parameter $\epsilon$. We first state an elementary lemma (without proof) for the symbol of $\mathcal{S}_\epsilon$:
\begin{lem}
\label{lem:multiplier}
 We have $\frac{1}{\epsilon} \sigma_\epsilon(\xi) \le \abs{\xi}$ for all $\xi \in \setR$ and
 $\epsilon >0$.
\end{lem}
Next, we use the rescaled energy functional for the phase function
and the corresponding Euler-Lagrange equation to obtain the required a priori estimates.
\begin{lem}
\label{lem:apriori}
 For the phase function $\theta_\epsilon$ of a rescaled N{\'e}el wall, the following a priori estimates
 are satisfied with a constant $C>0$ independent of $\epsilon>0$:
\begin{enumerate}
 \item[(i)] $\norm{\theta'_\epsilon}_{L^2}$, $\norm{\cos \theta_\epsilon}_{H^1} \le C$
 \item[(ii)] $\norm{\theta'_\epsilon}_{H^1}$, $\norm{\theta'_\epsilon}_{L^\infty} \le C$
 \item[(iii)] $\frac{1}{\epsilon}\norm{\mathcal{S}_\epsilon[\cos \theta_\epsilon] \cos \theta_\epsilon}_{L^\infty}
 \le C$
\end{enumerate}
\end{lem}
\begin{proof}
 For (i) we choose a smooth and admissible comparison function $\theta$ to find
\begin{align*}
 \kappa \int \abs{\theta'_\epsilon}^2 + \int \cos^2 \theta_\epsilon
 \le {E_\text{res}^\epsilon} (\theta_\epsilon)
 \le {E_\text{res}^\epsilon} (\theta)
 =\kappa \int \abs{\theta'}^2 + \int \cos^2 \theta + \frac{1}{\epsilon} \int \mathcal{S}_\epsilon [\cos \theta]
 \cos \theta \,.
\end{align*}
The definition of $\mathcal{S}_\epsilon$ and Lemma \ref{lem:multiplier} lead to
\begin{align*}
 \kappa \int \abs{\theta'_\epsilon}^2 + \int \cos^2 \theta_\epsilon
 \le\kappa \int \abs{\theta'}^2 + \int \cos^2 \theta + \int \abs{\xi}
 \abs{\widehat{\cos \theta}}^2
 \le \kappa \norm{\theta'}_{L^2}^2 + 2 \norm{\cos \theta}_{H^1}^2
 \le C
\end{align*}
with some constant $C>0$ independent of $\epsilon>0$. Hence, we obtain the estimates $\norm{\theta'_\epsilon}_{L^2}$, $\norm{\cos \theta_\epsilon}_{L^2} \le C$. Since $(\cos \theta_\epsilon)' = - \theta'_\epsilon \sin \theta_\epsilon$, we also find $\norm{\cos \theta_\epsilon}_{H^1} \le C$.

For (ii) we recall that $\theta_\epsilon$ solves the Euler-Lagrange equation \ref{eq:Euler-Lagrange-N}:
\begin{align*}
 \kappa \theta''_\epsilon = - \sin \theta_\epsilon \cos \theta_\epsilon - \frac{1}{\epsilon} \mathcal{S}_\epsilon
 [\cos \theta_\epsilon] \sin \theta_\epsilon \,.
\end{align*}
From the first part, we get $\norm{\sin \theta_\epsilon \cos \theta_\epsilon}_{L^2} \le \norm{\cos \theta_\epsilon}_{L^2} \le C$. Moreover, we can estimate the remaining term on the right hand side with the help of Lemma \nolinebreak \ref{lem:multiplier} and (i) as follows:
\begin{align*}
 \frac{1}{\epsilon^2} \norm{\mathcal{S}_\epsilon [\cos \theta_\epsilon] \sin \theta_\epsilon}_{L^2}^2 \le
 \frac{1}{\epsilon^2} \int \mathcal{S}_\epsilon [\cos \theta_\epsilon] ^2
 \le \int \abs{\xi}^2 \abs{\widehat{\cos \theta_\epsilon}}^2
 \le \norm{\cos \theta_\epsilon}_{H^1}^2
 &\le C \,.
\end{align*}
We conclude $\norm{\theta''_\epsilon}_{L^2} \le C$.
This combined with (i) and the embedding $H^1 \hookrightarrow L^\infty$ implies (ii).

For (iii) we first remark that
 $\,(\cos \theta_\epsilon)'' = (- \theta'_\epsilon \sin \theta_\epsilon )' = - \theta''_\epsilon \sin \theta_\epsilon
 - (\theta'_\epsilon)^2 \cos \theta_\epsilon\,$
and find the estimate $\norm{(\cos \theta_\epsilon)''}_{L^2} \le \norm{\theta''_\epsilon}_{L^2} + \norm{\theta'_\epsilon}_{L^\infty}^2 \norm{\cos \theta_\epsilon}_{L^2} \le C$ thanks to (i) and (ii). In particular, we have $\norm{\cos \theta_\epsilon}_{H^2} \le C$. We use this and Lemma \ref{lem:multiplier} to see
\begin{align*}
 \frac{1}{\epsilon^2} \norm{\mathcal{S}_\epsilon [\cos \theta_\epsilon]}_{H^1}^2 &= \frac{1}{\epsilon^2} 
 \int (1+ \abs{\xi}^2 ) \,
 \abs{\sigma_\epsilon (\xi) \, \widehat{\cos \theta_\epsilon}(\xi)}^2
 \le \int (1+ \abs{\xi}^2)^2 \, \abs{\widehat{\cos \theta_\epsilon}(\xi)}^2
 = \norm{\cos \theta_\epsilon}_{H^2}^2 \, ,
\end{align*}
hence $\frac{1}{\epsilon} \norm{\mathcal{S}_\epsilon [\cos \theta_\epsilon]}_{H^1} \le C$.
It follows $\frac{1}{\epsilon} \norm{\mathcal{S}_\epsilon [\cos \theta_\epsilon] \cos \theta_\epsilon}_{L^2} \le C$ and
\begin{align*}
 \frac{1}{\epsilon} \norm{(\mathcal{S}_\epsilon [\cos \theta_\epsilon] \cos \theta_\epsilon)'}_{L^2} &\le
 \frac{1}{\epsilon} \norm{\mathcal{S}_\epsilon [\cos \theta_\epsilon]}_{H^1} + \frac{1}{\epsilon} 
 \norm{\mathcal{S}_{\epsilon}[\cos \theta_\epsilon] \theta'_\epsilon \sin \theta_\epsilon}_{L^2}
 \\
 &\le C + \frac{1}{\epsilon} \norm{\mathcal{S}_\epsilon[\cos \theta_\epsilon]}_{L^2}
 \norm{\theta'_\epsilon}_{L^\infty}
 \\
 &\le C \,.
\end{align*}
We end up with $\frac{1}{\epsilon}\norm{\mathcal{S}_\epsilon [\cos \theta_\epsilon] \cos \theta_\epsilon}_{L^\infty} \le \frac{C}{\epsilon} \, \norm{\mathcal{S}_\epsilon [\cos \theta_\epsilon] \cos \theta_\epsilon}_{H^1} \le C$. This proves (iii) and the lemma.
\end{proof}
With the help of Lemma \ref{lem:apriori}, we can show that $\mathcal{L}_1^\epsilon$ is invertible for $\epsilon$ small enough.
\begin{lem}
 \label{lem:L1 is sectorial}
 The linear operator $\mathcal{L}_1^\epsilon : H^2 \subset L^2 \to L^2$ is sectorial and self-adjoint. Furthermore, $\mathcal{L}_1^\epsilon$ is invertible for
 $\epsilon>0$ small enough.
\end{lem}
\begin{proof}
 Because of the decomposition $\mathcal{L}_1^\epsilon =  \kappa \Delta + \mathcal{B}^\epsilon$
with a linear and bounded operator $\mathcal{B}^\epsilon:L^2 \to L^2$, we obtain from \cite[Proposition 2.4.1]{lunardi} (i) that $\mathcal{L}_1^\epsilon$ is sectorial. In particular, there are $\lambda_1$ and $\lambda_2$ in $\rho(\mathcal{L}_1^\epsilon)$ such that $\text{Im}\lambda_1>0$ and $\text{Im}\lambda_2<0$. This combined with the fact that $\mathcal{L}_1^\epsilon$ is $L^2$-symmetric implies that $\mathcal{L}_1^\epsilon$ is self-adjoint. To prove the remaining statement, we define
\begin{align*}
 \mathcal{G}(u,v) = \langle-\mathcal{L}_1^\epsilon u , v\rangle
 =& \kappa \int u'v' + \frac{1}{\epsilon} \int uv + \frac{1}{2} 
 \int uv
 - \frac{1}{2} \int \cos(2\theta_\epsilon) uv 
 \\
 &
 - \kappa \int (\theta'_\epsilon)^2 u v - \frac{1}{\epsilon} 
 \int \mathcal{S}_\epsilon [\cos \theta_\epsilon] \cos \theta_\epsilon \,u v - \frac{1}{\epsilon}
 \int\mathcal{S}_\epsilon [u] v
\end{align*}
for $u,v \in H^1$. Then $\mathcal{G} : H^1 \times H^1 \to \setR$ is well-defined, bilinear, and bounded. Moreover, we can estimate with the help of Lemma \ref{lem:apriori} as follows:
\begin{align*}
 \mathcal{G}(u,u) 
 \ge& \kappa \int\abs{u'}^2 + \frac{1}{\epsilon} \int \abs{u}^2 - C \int \abs{u}^2 -
 \frac{1}{\epsilon} \int \sigma_\epsilon \abs{\widehat{u}}^2 \, .
\end{align*}
We obtain by applying Lemma \ref{lem:multiplier} and the Young inequality that
\begin{align*}
 \mathcal{G}(u,u) 
 \ge \kappa \int\abs{u'}^2 + \frac{1}{\epsilon} \int \abs{u}^2 - C \int \abs{u}^2
 - \frac{1}{2\kappa} \int \abs{u}^2 - \frac{\kappa}{2} \int \abs{u'}^2
 = \frac{\kappa}{2} \int\abs{u'}^2 + \frac{1}{\epsilon} \int \abs{u}^2 - C \int \abs{u}^2
\end{align*}
with some constant $C>0$ independent of $\epsilon>0$. In particular, $\mathcal{G}$ becomes coercive for $\epsilon$ small enough. Thanks to the Lax-Milgram theorem, we find for every $f \in L^2$ a unique element $u \in H^1$ such that $\mathcal{G}(u,v) = -(f,v)_{L^2}$ for all $v \in H^1$. This means that $u$ is the unique weak solution of $\mathcal{L}_1^\epsilon u = f$, and from here we directly obtain $u \in H^2$, thus $\mathcal{L}_1^\epsilon$ is invertible. The lemma is proved.
\end{proof}

\section{The linear operator $\mathcal{L}_2^\epsilon$}
\label{sec:NL2}
In this section we analyze the operator $\mathcal{L}_2^\epsilon$. As already remarked, $\mathcal{L}_2^\epsilon$ is the linearization of the rescaled Euler-Lagrange equation for the phase function at $\theta_\epsilon$. Due to the translation invariance of the energy functional, we expect $\mathcal{L}_2^\epsilon$ to have a kernel of at least dimension one. This is true as shown in the next lemma:
\begin{lem}
 \label{lem:L2 is sectorial}
 The linear operator $\mathcal{L}_2^\epsilon$ is sectorial and self-adjoint. Moreover, the function
 $\theta'_\epsilon$ belongs to the kernel of $\mathcal{L}_2^\epsilon$.
\end{lem}
\begin{proof}
 As in Lemma \ref{lem:L1 is sectorial}, we see that  $\mathcal{L}_2^\epsilon$ 
 is sectorial and self-adjoint. We also know that $\theta_\epsilon$ is 
 smooth and 
 solves the Euler-Lagrange equation \ref{eq:Euler-Lagrange-N}.
 Furthermore, we have thanks to Lemma \ref{lem:Nphasefunction} that $\cos \theta_\epsilon \in H^k$ for all $k \in \setN$, hence
 $\mathcal{S}_\epsilon[\cos \theta_\epsilon] \in H^k$ for all $k \in \setN$. The Sobolev embedding
 theorem implies smoothness of $\mathcal{S}_\epsilon[\cos \theta_\epsilon]$, and with the help of the Fourier
 transform, we obtain the identity
 \begin{align*}
  \big( \mathcal{S}_\epsilon [\cos \theta_\epsilon] \big) ' = \mathcal{S}_\epsilon [ (\cos \theta_\epsilon)' ]
  = -\mathcal{S}_\epsilon [ \sin \theta_\epsilon \, \theta'_\epsilon ] \, .
 \end{align*}
  We now differentiate the Euler-Lagrange equation \ref{eq:Euler-Lagrange-N} with respect to the space variable and obtain
 \begin{align*}
  0 &= \kappa \theta'''_\epsilon + \cos(2 \theta_\epsilon) \theta'_\epsilon -\frac{1}{\epsilon}
   \mathcal{S}_\epsilon
  [\sin \theta_\epsilon \, \theta'_\epsilon] \sin \theta_\epsilon + \frac{1}{\epsilon} \mathcal{S}_\epsilon 
  [\cos \theta_\epsilon] \cos \theta_\epsilon \, \theta'_\epsilon
  = \mathcal{L}_2^\epsilon \theta'_\epsilon \, .
 \end{align*}
 We already know that $\theta'_\epsilon \in H^2$, hence $\theta'_\epsilon \in N(\mathcal{L}_2^\epsilon)$. The
 lemma is proved.
\end{proof}
In the sequel we show that the kernel of $\mathcal{L}_2^\epsilon$ is 
actually one-dimensional and therefore given by $\text{span}\set{\theta'_\epsilon}$. For this we prove
a spectral gap estimate for $\mathcal{L}_2^\epsilon$ and follow the 
arguments presented in \cite{CMO}. Due to the reduction made for LLG in 
\cite{CMO}, the symbol of $\mathcal{S}_\epsilon$ differs from the 
one we have in our situation. However, this requires only minor changes. To keep our presentation self-contained, we have decided to repeat the proof here. We start by defining
\begin{align*}
 \mathcal{G}(u,v) = \langle-\mathcal{L}_2^\epsilon u,v\rangle
 = \kappa \int u' v' -\int \cos (2 \theta_\epsilon) u v + \frac{1}{\epsilon} \int \mathcal{S}_\epsilon[
 \sin \theta_ \epsilon \, u] \sin \theta_\epsilon \, v 
 - \frac{1}{\epsilon} \int \mathcal{S}_\epsilon [\cos \theta_\epsilon] \cos \theta_\epsilon \, uv
\end{align*}
and
 $\,\mathcal{H}(u,v) = \int \big( 1 + \epsilon^{-1} \sigma_\epsilon \big) \widehat{u} \, \overline{
 \widehat{v}}\,$
for all $u,v \in H^1$. The next lemma is the major step towards the spectral gap 
estimate and uses the rescaled Euler-Lagrange equation \ref{eq:Euler-Lagrange-N} together 
with a clever chosen test function.
\begin{lem}
\label{lem:estimate for G}
 For all $u \in H^1$ with $u(0)=0$, we have the estimate
\begin{align*}
 \mathcal{G}(u,u) \ge \kappa \norm{u \theta'_\epsilon}_{L^2}^2 + \mathcal{H}(u \sin \theta_\epsilon , u \sin \theta_\epsilon) \, .
\end{align*}
\end{lem}
\begin{proof}
 First, we assume that $u=0$ in a neighborhood of $0$ and rewrite $\mathcal{G}(u,u)$ with the help of 
 the identity
  $\cos(2\theta_\epsilon) = \cos^2 \theta_\epsilon - \sin^2 \theta_\epsilon$
 in the following
 way:
 \begin{align*}
  \mathcal{G}(u,u) 
  =& \kappa \int \abs{u'}^2 - \int \cos \theta_\epsilon \, \abs{u}^2 \big(\cos \theta_\epsilon + \epsilon^{-1}
  \mathcal{S}_\epsilon [\cos \theta_\epsilon] \big)
  + \int \sin \theta_\epsilon \, u \big( \sin \theta_\epsilon \, u + \epsilon^{-1} \mathcal{S}_\epsilon
  [\sin \theta_\epsilon \, u] \big)
  \\
  =& \kappa \int \abs{u'}^2 - \mathcal{H}(\cos \theta_\epsilon \, \abs{u}^2, \cos \theta_\epsilon) +
 \mathcal{H}( \sin \theta_\epsilon \, u, \sin \theta_\epsilon \, u) \, .
 \end{align*}
 In order to estimate the first two terms on the right hand side, we rewrite the weak form of
 the rescaled Euler-Lagrange equation \ref{eq:Euler-Lagrange-N} for the phase function in a similar way as above. To be more precise, we have
 \begin{align*}
  0&= -\kappa \int \theta'_\epsilon v' + \frac{1}{2} \int \sin (2 \theta_\epsilon) v + \frac{1}{\epsilon} 
  \int \mathcal{S}_\epsilon [ \cos \theta_\epsilon ] \sin \theta_\epsilon \, v
  = - \kappa \int \theta'_\epsilon v' + \mathcal{H}( \sin \theta_\epsilon \, v ,\cos \theta_\epsilon ) \,
 \end{align*}
 for all $v \in H^1$. Since $\theta_\epsilon (x) = 0$ if and only if $x = 0$ (see Lemma \ref{lem:Nphasefunction}), we can introduce the test
 function $v = u^2 \cot \theta_\epsilon \in H^1$. Inserting leads to
 \begin{align*}
  \mathcal{H}(\cos \theta_\epsilon \, u^2, \cos \theta_\epsilon ) 
  &= 2 \kappa \int u u' \cot \theta_\epsilon \, \theta'_\epsilon - \kappa \int \abs{\theta'_\epsilon}^2 u^2 
  ( 1 + \cot^2 \theta_\epsilon) \,.
 \end{align*}
 This together with the Young inequality implies
 \begin{align*}
  \kappa \hspace{-0.1cm}\int \abs{u'}^2 - \hspace{-0.1cm}\mathcal{H}(\cos \theta_\epsilon \, \abs{u}^2, \cos \theta_\epsilon ) 
  \ge& \kappa \hspace{-0.1cm}\int \abs{u'}^2 + \hspace{-0.1cm}\kappa \hspace{-0.1cm}\int \abs{\theta'_\epsilon}^2 \abs{u}^2 + \hspace{-0.1cm}\kappa \hspace{-0.1cm}\int \abs{\theta'_\epsilon}^2
  \abs{u}^2 \cot^2 \theta_\epsilon - \hspace{-0.1cm}\kappa \hspace{-0.1cm}\int \abs{u'}^2 - \hspace{-0.1cm} \kappa \hspace{-0.1cm}\int \abs{u}^2 \abs{\theta'_\epsilon}^2 \cot^2
  \theta_\epsilon
  \\
  =& \kappa \hspace{-0.1cm}\int \abs{\theta'_\epsilon}^2 \abs{u}^2 \, .
 \end{align*}
 A combination of the above estimates yields
 \begin{align*}
  \mathcal{G}(u,u) \ge \kappa \int \abs{\theta'_\epsilon}^2 \abs{u}^2 + \mathcal{H} 
 (\sin \theta_\epsilon \, u,\sin \theta_\epsilon \, u)
 \end{align*}
 for all $u \in H^1$ with $u=0$ in a neighborhood of $0$. To complete the proof, let $u \in H^1$
 with $u(0)=0$ be given and define for $\delta >0$ the truncated function $u_\delta$ by  $u_\delta(x)=u(x- \delta)$ if $x \ge \delta$, $u_\delta (x)=0$ if 
 $-\delta <x< \delta$, and $u_\delta (x) = u(x+\delta)$ if $x \le -\delta$.
 Since $u_\delta \to u$ in $H^1$ for $\delta \to 0$, the lemma follows from the previous inequality by means of approximation.
\end{proof}
We are now in a position to prove the spectral gap estimate for 
$\mathcal{L}_2^\epsilon$.
\begin{lem}
 \label{lem:L2 is coerciv}
 There is a constant $C=C(\epsilon)>0$ such that
  $\mathcal{G}(u,u) \ge C \norm{u}_{L^2}^2$
 for all $u \in H^1$ with $u\perp \theta'_\epsilon$.
\end{lem}
\begin{proof}
 For $u \in H^1$ with $u\perp \theta'_\epsilon$, we consider $v = u - t \theta'_\epsilon \in H^1$ where
 $t = u(0)/\theta'_\epsilon(0)$ and remark that
 \begin{align*}
  \mathcal{G}(v,v) 
  = \mathcal{G}(u,u) + 2 t (\mathcal{L}_2^\epsilon \theta'_\epsilon,u)_{L^2} - t^2 (\mathcal{L}_2^\epsilon
  \theta'_\epsilon, \theta'_\epsilon)_{L^2}
  = \mathcal{G}(u,u)
 \end{align*}
 thanks to Lemma \ref{lem:L2 is sectorial}. Moreover, Lemma 
\ref{lem:estimate for G} together with the fact
 $v(0)=0$ implies
 \begin{align*}
  \mathcal{G}(u,u) 
  \ge \kappa \int \abs{v}^2 \abs{\theta'_\epsilon}^2 + \int \abs{v}^2 \sin^2 \theta_\epsilon
  \ge \min \set{\kappa,1} \int \abs{v}^2 \big( \abs{\theta'_\epsilon}^2 + \sin^2 \theta_\epsilon \big)
  \ge C(\epsilon) \int \abs{v}^2 \, .
 \end{align*}
 In the previous line, we have used that $\theta_\epsilon$ is nondecreasing and $\theta'_\epsilon (0)>0$ (see Lemma \ref{lem:Nphasefunction}). The assumption $u\perp \theta'_\epsilon$ implies the statement of the lemma.
\end{proof}
Next, the spectral gap estimate is used to determine the range of
$\mathcal{L}_2^\epsilon$.
\begin{lem}
\label{lem:L2 is invertible on the complement of y}
 For all $f \in L^2$ with $f \perp \theta'_\epsilon$ there exists a unique $u \in H^2$ with $u \perp \theta'_\epsilon$ such that $\mathcal{L}_2^\epsilon u = f$. Moreover, we have the a priori estimate $\norm{u}_{H^2} \le C \, \norm{f}_{L^2}$ with a constant $C=C(\epsilon)>0$.
\end{lem}
\begin{proof}
 First, we show uniqueness. Let therefore $u_1,u_2 \perp \theta'_\epsilon$ be given such that
 $\mathcal{L}_2^\epsilon u_1 = \mathcal{L}_2^\epsilon u_2 =f$. We find $\mathcal{G}(u_1 - u_2,v) = 0$ for all $v \in H^1$, and thanks to Lemma 
\ref{lem:L2 is coerciv}, we get
 \begin{align*}
  0= \mathcal{G}(u_1 - u_2, u_1 - u_2 ) \ge C \norm{u_1 - u_2}_{L^2}^2 \, ,
 \end{align*}
hence $u_1 = u_2$. Next, we show the existence of solutions for a given $f \in L^2$ with $f \perp
\theta'_\epsilon$ and define the space
 $H^1_{\perp} = \set{u \in H^1 \, | \, u \perp \theta'_\epsilon \text{ in } L^2}$.
We remark that $H^1_{\perp}$ is a closed subspace of $H^1$ and therefore a Hilbert space. We consider the
restriction of $\mathcal{G}$ on $H^1_{\perp} \times H^1_{\perp}$, which is again bilinear, symmetric, and bounded. Moreover, we obtain with Lemma \ref{lem:L2 is coerciv} the estimate
\begin{align*}
 (1+\delta) \mathcal{G}(u,u) = 
 \ge& C \norm{u}_{L^2}^2 + \delta \kappa \norm{u'}_{L^2}^2 - \delta \mathcal{H}(\cos \theta_\epsilon \, \abs{u}^2,
 \cos \theta_\epsilon) + \delta \mathcal{H}( \sin \theta_\epsilon \, u , \sin \theta_\epsilon \, u) 
\end{align*}
for all $u \in H^1_{\perp}$ and some $\delta>0$ to be chosen below. Because of Lemmas \ref{lem:reduced stray field} and \ref{lem:Nphasefunction}, we have that
\begin{align*}
  \mathcal{H}(\cos \theta_\epsilon \, \abs{u}^2, \cos \theta_\epsilon) =
 \int \cos^2 \theta_\epsilon \, \abs{u}^2 + \frac{1}{\epsilon} \int \cos \theta_\epsilon \, \abs{u}^2 \mathcal{S}_\epsilon[\cos \theta_\epsilon] \le C \norm{u}_{L^2}^2 \,.
\end{align*}
Since $\sigma_\epsilon \ge 0$, we see $\mathcal{H}(\sin \theta_\epsilon \, u, \sin \theta_\epsilon \,u) \ge 0$ and therefore
$\mathcal{G}(u,u) \ge c \norm{u}_{H^1}^2$ for all 
$u \in H^1_{\perp}$, provided $\delta$ is chosen small enough. This means that $\mathcal{G}$ is coercive on $H^1_{\perp}$, and with the help of the Lax-Milgram theorem, we find a unique $u \in H^1_{\perp}$ such that
 $\mathcal{G}(u,v) = \langle-\mathcal{L}_2^\epsilon u,v\rangle = (-f,v)_{L^2}$
for all $v \in H^1_{\perp}$ and $\norm{u}_{H^1}\le C\, \norm{f}_{L^2}$. In the sequel we prove that this actually holds for all $v \in H^1$ and define therefore the projection $P$ by
 $Pv = v - (v,\theta'_\epsilon)_{L^2}/\norm{\theta'_\epsilon}_{L^2}^2 \,  \theta'_\epsilon = v - t(v) \theta'_\epsilon$
for all $v \in L^2$. The projection $P$ has the following properties: $Pv \perp \theta'_\epsilon$, $v \perp \theta'_\epsilon$ implies $Pv=v$, and $(Pv,w)_{L^2} = (v,Pw)_{L^2}$ for all $v,w \in L^2$.
Furthermore, it holds $\mathcal{G}(v,Pw) = \mathcal{G}(v,w)$ for all $v,w \in H^1$ since
\begin{align*}
 \mathcal{G}(v,Pw) &= 
\mathcal{G}(v,w) - t(w) \mathcal{G}(v,\theta'_\epsilon) = \mathcal{G}(v,w) + t(w) (\mathcal{L}_2^\epsilon \theta'_\epsilon,v)_{L^2} =
 \mathcal{G}(v,w) \, .
\end{align*}
For $v \in H^1$ we have $Pv \in H^1_{\perp}$ and therefore
\begin{align*}
 \mathcal{G}(u,v) = \mathcal{G}(u,Pv) = (-f,Pv)_{L^2} = (-Pf,v)_{L^2} = (-f,v)_{L^2}
\end{align*}
for all $v \in H^1$. We conclude $u \in H^2$, $\mathcal{L}_2^\epsilon u =f$, and $\norm{u}_{H^2} \le C \, \norm{f}_{L^2}$. The lemma is proved.
\end{proof}
Finally, we summarize the properties of the self-adjoint operator $\mathcal{L}_2^\epsilon$ in the following lemma.
\begin{lem}
\label{lem:properties of L2}
 The following statements for $\mathcal{L}_2^\epsilon$ hold true:
 \begin{enumerate}
  \item[(i)] $N(\mathcal{L}_2^\epsilon) = \text{span} \set{\theta'_\epsilon}$
  \item[(ii)] $R(\mathcal{L}_2^\epsilon) = L^2_{\perp} = \set{ f \in L^2 \, | \, f \perp \theta'_\epsilon}$
  \item[(iii)] $\mathcal{L}_2^\epsilon : H^2_{\perp} \to L^2_{\perp}$ is an isomorphism, where $H^2_\perp =
 \set{u \in H^2 \, | \, u \perp \theta'_\epsilon}$.
 \end{enumerate}
\end{lem}

\section{Spectral analysis for $\mathcal{L}_0^\epsilon$}
\label{sec:NL0}
In this section we combine the results of Sections \ref{sec:NL1} and \ref{sec:NL2} to study the properties of the linearization \nolinebreak $\mathcal{L}_0^\epsilon$. In particular, we show that $0$ is an isolated point in $\sigma(\mathcal{L}_0^\epsilon)$ (in fact we show that $0$ is semisimple) and $\sigma(\mathcal{L}_0^\epsilon) \cap \dot{\imath}\setR = \set{0}$. This has important consequences for the analytic semigroup generated by the sectorial operator $\mathcal{L}_0^\epsilon$ and is the crucial ingredient for our perturbation argument in Section \ref{sec:Nperturbation}. We start by showing that the leading order term of $\mathcal{L}_0^\epsilon$ defines a sectorial operator.
\begin{lem}
\label{lem:Laplacian is sectorial}
  For all $\alpha \in \setR$, the linear operator $\mathcal{L}$ defined by
  \begin{align*}
    \mathcal{L} =
    \begin{pmatrix}
      \hphantom{-\alpha} \Delta & \alpha \Delta
      \\
      - \alpha \Delta & \hphantom{\alpha} \Delta
    \end{pmatrix}
    : H^2 \times H^2 \subset L^2 \times L^2 \to L^2 \times L^2
  \end{align*}
  is sectorial, and the graph norm of $\mathcal{L}$ is equivalent to the $H^2$-norm.
\end{lem}
\begin{proof}
 For a given $\lambda \in \setC$ with $\text{Re}\lambda >0$, we define $\mathcal{G}$ by
  \begin{align*}
    \mathcal{G}(u,v) =& \langle\lambda u - \mathcal{L}u,v\rangle
    = \lambda \int u_1 \overline{v_1} + \lambda \int u_2 \overline{v_2} + \int u_1' \overline {v_1'} + \alpha \int u_2' \overline{v_1'} 
    - \alpha \int u_1' \overline{v_2'} + \int u_2' \overline{v_2'}
  \end{align*}
  for $u=(u_1,u_2),v=(v_1,v_2) \in H^1 \times H^1$.
  We see that $\mathcal{G}$ is well-defined, sesquilinear, and bounded. 
  Moreover, $\mathcal{G}$ is coercive since
  \begin{align*}
    \text{Re} \,\mathcal{G} (u,u) =& \text{Re} \lambda \norm{u_1}_{L^2}^2 + \text{Re} \lambda \norm{u_2}_{L^2}^2 + \norm{u_1'}_{L^2}^2 + \norm{u_2'}_{L^2}^2
    = \text{Re} \lambda  \norm{u}_{L^2}^2 + \norm{u'}_{L^2}^2
  \end{align*}
  and $\text{Re} \lambda >0$. With the help of the classical Lax-Milgram theorem,
  we find for every $f = (f_1,f_2) \in L^2 \times L^2$ 
  a unique $u = (u_1,u_2) \in H^1 \times H^1$ such that
    $\mathcal{G}(u,v) = (f,v)_{L^2}$ for all $v=(v_1,v_2) \in H^1 \times H^1$.
  In particular, $u$ is the unique weak solution of the resolvent equation 
    $\lambda u - \mathcal{L} u = f$.
  This also shows that 
  $u_1$ and $u_2$ satisfy
  \begin{align*}
    \lambda (u_1 - \alpha u_2) - (1 + \alpha^2)\Delta u_1 &= f_1 -\alpha f_2
    \\
    \lambda (u_2 + \alpha u_1) - (1 + \alpha^2)\Delta u_2 &= \alpha f_1 +f_2
  \end{align*}
  in the sense of distributions, hence $u_1,u_2 \in H^2$. We conclude that the resolvent equation admits a unique solution 
  $u \in H^2 \times H^2$ for every right hand side $f \in L^2 \times L^2$. For the resolvent estimate, let $u=(u_1,u_2) \in H^2 \times H^2$ and
  $f = (f_1,f_2) \in L^2 \times L^2$ with $\lambda u - \mathcal{L} u = f$ be given. We find
  \begin{align*}
  (f,\Delta u)_{L^2} =& (\lambda u - \mathcal{L} u ,\Delta u )_{L^2}
  \\
  =& - \lambda \norm{u'_1}_{L^2}^2  - \norm{\Delta u_1}_{L^2}^2 - \alpha (\Delta u_2, \Delta u_1)_{L^2}
   -  \lambda \norm{u'_2}_{L^2}^2   +  \alpha (\Delta u_1,\Delta u_2)_{L^2} - \norm{\Delta u_2}_{L^2}^2   \,.
  \end{align*}
  From here we obtain by considering only the real part that
  \begin{align*}
    \text{Re}\lambda \, \norm{u'}_{L^2}^2 + \norm{\Delta u}_{L^2}^2 = - \text{Re} (f,\Delta u)_{L^2} \le \norm{f}_{L^2} \, 
    \norm{\Delta u}_{L^2} \, ,
  \end{align*}
  hence $\norm{\Delta u}_{L^2} \le \norm{f}_{L^2}$. In particular, we have
  \begin{align*}
    \abs{\lambda} \, \norm{u}_{L^2} \le \norm{f}_{L^2} + \norm{\mathcal{L} u}_{L^2} \le   \norm{f}_{L^2} + C_\alpha 
    \norm{\Delta u}_{L^2} \le C_\alpha \, \norm{f}_{L^2} \, ,
  \end{align*}
  thus $\norm{\lambda \text{R}(\lambda,\mathcal{L})} \le C_\alpha$ for all $\text{Re}\lambda >0$. Proposition \ref{prop:sectorial}
  implies that $\mathcal{L}$ is sectorial. Furthermore, we see with the help of the open mapping theorem that the graph norm of $\mathcal{L}$ is equivalent to the $H^2$-norm. The lemma is proved.
\end{proof}
Above we made use of the following proposition which is taken from \cite[Proposition 2.1.11]{lunardi}.
\begin{proposition}
\label{prop:sectorial}
 Let $A:D(A) \subset X \to X$ be a linear operator on a Banach space $X$ such that the resolvent set $\rho(A)$ contains the half-plane $\set{\lambda \in \setC \, | \, \text{Re} \lambda \ge \omega}$ 
and 
 $\norm{\lambda \text{R}(\lambda,A)} \le M$ for all $\text{Re} \lambda \ge \omega$,
 where $\omega \in \setR$ and $M>0$. Then $A$ is sectorial.
\end{proposition}
Since we have the decomposition $\mathcal{L}_0^\epsilon =\kappa \mathcal{L}  + \mathcal{B}^\epsilon$ where $\mathcal{L}$ is sectorial and $\mathcal{B}^\epsilon:L^2\times L^2 \to L^2\times L^2$ is bounded,
we immediately obtain that $\mathcal{L}_0^\epsilon$ is sectorial. In the sequel we show that \nolinebreak $0$ is an isolated point in $\sigma(\mathcal{L}_0^\epsilon)$. Therefore, we first identify the kernel and range of $\mathcal{L}_0^\epsilon$.
\begin{lem}
\label{lem:L0 splitting}
 Let $\epsilon>0$ be small enough. Then the following statements for the linear operator $\mathcal{L}_0^\epsilon$ hold true:
\begin{enumerate}
 \item[(i)] $N(\mathcal{L}_0^\epsilon) = \set{0} \times N(\mathcal{L}_2^\epsilon) = \set{0} \times \text{span}\set{\theta'_\epsilon}$.
 \item[(ii)]  $R(\mathcal{L}_0^\epsilon)$ is closed and $R(\mathcal{L}_0^\epsilon) = 
 \set{ (f_1,f_2) \in L^2 \times L^2 \, | \, (\alpha f_1 +f_2 , \theta'_\epsilon)_{L^2} = 0}$.
 \item[(iii)] $L^2 \times L^2 = N(\mathcal{L}_0^\epsilon) \oplus R(\mathcal{L}_0^\epsilon)$.
\end{enumerate}
\end{lem}
\begin{proof}
 Statement (i) follows directly from the equivalences
\begin{align*}
 (u_1,u_2) \in N(\mathcal{L}_0^\epsilon) \quad &\Leftrightarrow \quad
 0= \mathcal{L}_1^\epsilon u_1 + \alpha \mathcal{L}_2^\epsilon u_2
 \quad \text{and} \quad
 0= -\alpha \mathcal{L}_1^\epsilon u_1 + \mathcal{L}_2^\epsilon u_2
\\
\quad&\Leftrightarrow\quad
 0= (1+\alpha^2) \mathcal{L}_1^\epsilon u_1
 \quad \text{and} \quad
 0= (1+\alpha^2) \mathcal{L}_2^\epsilon u_2
\\
\quad&\Leftrightarrow\quad
 u_1 = 0 \quad \text{and} \quad u_2 \in N(\mathcal{L}_2^\epsilon) = \text{span}\set{\theta'_\epsilon} 
 \, ,
\end{align*}
where we have used Lemmas \ref{lem:L1 is sectorial} and \ref{lem:properties of L2}.

For statement (ii), let first the couple $(f_1,f_2) \in R(\mathcal{L}_0^\epsilon)$ be given. This means
\begin{align*}
 f_1 &=\mathcal{L}_1^\epsilon u_1 + \alpha \mathcal{L}_2^\epsilon u_2
 \,, \qquad
 f_2 =-\alpha \mathcal{L}_1^\epsilon u_1 + \mathcal{L}_2^\epsilon u_2
\end{align*}
for some $u_1, u_2 \in H^2$. We can equivalently rewrite this as
\begin{align*}
 f_1 - \alpha f_2 &= (1+\alpha^2)\mathcal{L}_1^\epsilon u_1
 \, , \qquad
 \alpha f_1 + f_2 = (1+\alpha^2)\mathcal{L}_2^\epsilon u_2 \, ,
\end{align*}
hence $(\alpha f_1 + f_2, \theta'_\epsilon)_{L^2}=0$. For the converse inclusion, let now $f_1,f_2 \in L^2$ be such that $(\alpha f_1 + f_2, \theta'_\epsilon)_{L^2}=0$. Thanks to Lemma \ref{lem:properties of L2}, there exists a $u_2 \in H^2$ such that $\alpha f_1 + f_2 = (1+\alpha^2)\mathcal{L}_2^\epsilon u_2$, and because of Lemma \nolinebreak \ref{lem:L1 is sectorial}, there exists a unique $u_1 \in H^2$ such that$f_1 - \alpha f_2 = (1+\alpha^2)\mathcal{L}_1^\epsilon u_1$. In particular, we have $\mathcal{L}_0^\epsilon (u_1,u_2) = (f_1,f_2)$, hence $(f_1,f_2) \in R(\mathcal{L}_0^\epsilon)$. This proves (ii).

We now show that the sum $N(\mathcal{L}_0^\epsilon) + R(\mathcal{L}_0^\epsilon)$ is direct. Let therefore $(f_1,f_2)$ be an element of the intersection $N(\mathcal{L}_0^\epsilon) \cap R(\mathcal{L}_0^\epsilon)$. We find
$(f_1,f_2) = (0, \lambda \theta'_\epsilon)$ and
 $0 = (\alpha f_1 + f_2 , \theta'_\epsilon)_{L^2} = \lambda \norm{\theta'_\epsilon}_{L^2}^2$,
hence $(f_1,f_2)=(0,0)$ and the sum is direct. Let now $(f_1,f_2) \in L^2 \times L^2$ be arbitrary. We can decompose $(f_1,f_2)$ as follows:
\begin{align*}
 (f_1,f_2) = (0,\lambda \theta'_\epsilon) + (f_1,f_2 - \lambda \theta'_\epsilon) \in N(\mathcal{L}_0^\epsilon) \oplus R(\mathcal{L}_0^\epsilon) \, ,
\end{align*}
where $\lambda = (\alpha f_1 +f_2,\theta'_\epsilon)_{L^2}/\norm{\theta'_\epsilon}_{L^2}^2$. This proves (iii) and the lemma.
\end{proof}
We can now prove:
\begin{lem}
\label{lem:0 is isolated}
 Let $\epsilon>0$ be small enough. Then 0 is an isolated point in the spectrum of $\mathcal{L}_0^\epsilon$.
\end{lem}
\begin{proof}
 Lemma \ref{lem:L0 splitting} implies that $0$ belongs to the resolvent set of the operator
 \begin{align*}
  \mathcal{L}_0^\epsilon : H^2 \times H^2 \cap R(\mathcal{L}_0^\epsilon) \subset  R(\mathcal{L}_0^\epsilon)  \to
  R(\mathcal{L}_0^\epsilon) \,.
 \end{align*}
 Since the resolvent set is always open, we find $r>0$ such that
\begin{align*}
 \lambda I - \mathcal{L}_0^\epsilon : H^2 \times H^2 \cap R(\mathcal{L}_0^\epsilon) \subset R(\mathcal{L}_0^\epsilon) \to R(\mathcal{L}_0^\epsilon)
\end{align*}
is invertible for all $\abs{\lambda} < r$. Let now $0<\abs{\lambda} < r$ be given. We prove that the mapping 
\begin{align*}
\lambda I - \mathcal{L}_0^\epsilon:H^2 \times H^2 \subset L^2 \times L^2 \to L^2 \times L^2 
\end{align*}
is invertible and thus $\lambda \in \rho(\mathcal{L}_0^\epsilon)$. First, we show that $\lambda I - \mathcal{L}_0^\epsilon$ is injective. So assume $\lambda u - \mathcal{L}_0^\epsilon u =0$. We decompose
 $u = v + w \in N(\mathcal{L}_0^\epsilon) \oplus \big(H^2 \times H^2 \cap
 R(\mathcal{L}_0^\epsilon) \big)$
to find
\begin{align*}
 0= \lambda v + \big( \lambda w - \mathcal{L}_0^\epsilon w \big)
\in N(\mathcal{L}_0^\epsilon) \oplus R(\mathcal{L}_0^\epsilon) \, ,
\end{align*}
hence $\lambda v=0$ and $\lambda w - \mathcal{L}_0^\epsilon w =0$. Since $0<\abs{\lambda} < r$, we obtain $v=w=0$ and $\lambda I - \mathcal{L}_0^\epsilon$ is injective. Next, we show that $\lambda I - \mathcal{L}_0^\epsilon$ is surjective: Let therefore $f \in L^2 \times L^2$ be given. Again, we use the decomposition
\begin{align*}
 f = g + h \in N(\mathcal{L}_0^\epsilon) \oplus R(\mathcal{L}_0^\epsilon)
\end{align*}
and find $w \in H^2 \times H^2 \cap R(\mathcal{L}_0^\epsilon)$ such that
 $\lambda w - \mathcal{L}_0^\epsilon w = h$.
Moreover, we choose $v = \lambda^{-1} g \in N(\mathcal{L}_0^\epsilon)$ and set $u = v + w$ to see $\lambda u - \mathcal{L}_0^\epsilon u = f$. It follows $\lambda \in \rho(\mathcal{L}_0^\epsilon)$ for $0<\abs{\lambda}<r$. In particular, 0 is an isolated point in the spectrum $\sigma(\mathcal{L}_0^\epsilon$). The lemma is proved.
\end{proof}
Lemmas \ref{lem:L0 splitting} and \ref{lem:0 is isolated} have the following consequences for the analytic semigroup generated by $\mathcal{L}_0^\epsilon$:
\begin{lem}
\label{lem:analytic semigroup of L0}
Let $\epsilon >0$ be small enough. Then the following statements
for the analytic semigroup $ e^{t \mathcal{L}_0^\epsilon}$ generated by $\mathcal{L}_0^\epsilon $ hold true:
\begin{enumerate}
 \item[(i)] $e^{t \mathcal{L}_0^\epsilon} u = u$ for all $u \in N(\mathcal{L}_0^\epsilon)$
 \item[(ii)] $e^{t \mathcal{L}_0^\epsilon} \big( R(\mathcal{L}_0^\epsilon) \big) \subset 
  R(\mathcal{L}_0^\epsilon)$
 \item[(iii)] $\sigma\big({e^{t \mathcal{L}_0^\epsilon}}_{|R(\mathcal{L}_0^\epsilon)} \big) \setminus \set{0} = e^{t \sigma(\mathcal{L}_0^\epsilon)\setminus \set{0}}$ for all $t>0$.
\end{enumerate}
\end{lem}
\begin{proof}
 Because of Lemma \ref{lem:0 is isolated}, the point 0 is isolated in
 $\sigma(\mathcal{L}_0^\epsilon)$, hence the sets $\sigma_1 = \set{0}$ and $ \sigma_2 =
 \sigma(\mathcal{L}_0^\epsilon)\setminus\set{0}$ are closed. Moreover, we can find $r >0$ 
 such that $\sigma_1 \subset B_r(0)$ and $\sigma_2 \cap \overline{B_r(0)} =
 \emptyset$. We parameterize the boundary of $B_r(0)$ by a curve $\gamma$, oriented
 counterclockwise, and define
 \begin{align*}
  P = \frac{1}{2 \pi \dot{\imath}} \int_\gamma \text{R}(\xi,\mathcal{L}_0^\epsilon) \, d \xi \,.
 \end{align*}
 From \cite[Proposition A.1.2]{lunardi} we know that $P$ is a projection such that $P(L^2 \times
 L^2) \subset H^2 \times H^2$. Moreover, we can decompose $L^2\times L^2$ by
 \begin{align*}
 L^2\times L^2 = X_1 \oplus X_2 \, , \quad X_1 = P(L^2 \times L^2)\,, \quad X_2 = (I-P)(L^2 \times L^2) \,,
\end{align*}
and this induces a splitting of the operator $\mathcal{L}_0^\epsilon$ as follows:
\begin{align*}
 &\mathcal{A}_1:X_1 \to X_1: u \mapsto \mathcal{L}_0^\epsilon u \, ,
 \\
 &\mathcal{A}_2:H^2 \times H^2 \cap X_2 \to X_2: u \mapsto \mathcal{L}_0^\epsilon u\,.
\end{align*}
Again from \cite[Proposition A.1.2]{lunardi} we get $\sigma(\mathcal{A}_1) = \sigma_1$, $\sigma(\mathcal{A}_2) = \sigma_2$, and $\text{R}(\lambda,\mathcal{A}_1) = \text{R}(\lambda, \mathcal{L}_0^\epsilon)_{|X_1}$, $\text{R}(\lambda,\mathcal{A}_2) = \text{R}(\lambda, \mathcal{L}_0^\epsilon)_{|X_2}$ for all $\lambda \in \rho(\mathcal{L}_0^\epsilon)$. In particular, $\mathcal{A}_1$ and $\mathcal{A}_2$ generate analytic semigroups on $X_1$ and $X_2$, respectively, and
we have
\begin{align*}
 e^{t \mathcal{A}_1} = {e^{t \mathcal{L}_0^\epsilon}}_{|X_1} \, , \quad e^{t \mathcal{A}_2} = {e^{t \mathcal{L}_0^\epsilon}}_{|X_2} \,.
\end{align*}
From Lemma \ref{lem:L0 splitting} and \cite[Proposition A.2.2]{lunardi}, we obtain $X_1 = N(\mathcal{L}_0^\epsilon)$ and $X_2 = R(\mathcal{L}_0^\epsilon)$. This proves (i) and (ii). To see (iii), we apply the spectral mapping theorem (see \cite[Corollary 2.3.7]{lunardi}) to $\mathcal{A}_2$ and find
\begin{align*}
 \sigma({e^{t \mathcal{L}_0^\epsilon}}_{|R(\mathcal{L}_0^\epsilon)}) \setminus \set{0}= \sigma(e^{t \mathcal{A}_2}) \setminus \set{0} = e^{t \sigma(\mathcal{A}_2)} = e^{t \sigma(\mathcal{L}_0^\epsilon)\setminus \set{0}}
\end{align*}
for every $t>0$. The lemma is proved.
\end{proof}
For the existence of $T$-periodic solutions, it is important 
to know 
whether or not the real number $1$ belongs to the spectrum of $e^{T\mathcal{L}_0^\epsilon}$.
The general spectral mapping theorem for 
analytic semigroups (see for example \cite[Corollary 2.3.7]{lunardi})
includes the following equivalence:
\begin{align*}
  1 \not\in \sigma(e^{T \mathcal{L}_0^\epsilon}) \quad \Leftrightarrow \quad 
  \frac{2 k \pi \dot{\imath}}{T} \not\in \sigma(\mathcal{L}_0^\epsilon) 
\text{ for all } k \in \setZ\,.
\end{align*}
We already know that $2 k \pi \dot{\imath}/T \in \sigma(\mathcal{L}_0^\epsilon)$ for $k=0$ because $\mathcal{L}_0^\epsilon$ has a nontrivial kernel. 
In view of Lemma \nolinebreak \ref{lem:analytic semigroup of L0}, it is 
good to analyze the remaining cases where $k \not= 0$. Since the operators 
$\mathcal{L}_1^\epsilon$ and $\mathcal{L}_2^\epsilon$ do not commute, we
can not use the spectral theorem for commuting self-adjoint operators to answer this question. It turns out that the self-adjointness of the operators $\mathcal{L}_1^\epsilon$ and $\mathcal{L}_2^\epsilon$ already suffices to completely rule out the cases where $k \not= 0$. This is the statement 
of the next lemma:
\begin{lem}
\label{lem:Spectrum of T}
 Let $H$ be a Hilbert space and $D \subset H$ be a dense 
subspace. Furthermore, let $\mathcal{A},\mathcal{B}$ be linear and self-adjoint operators from $D$ to $H$.
Consider for $\alpha \in \setR$ the linear operator $ \mathcal{T}$ defined by
\begin{align*}
 \mathcal{T} =
 \begin{pmatrix}
  \hphantom{-\alpha} \mathcal{A} & \alpha \mathcal{B}
  \\
  -\alpha \mathcal{A} & \hphantom{\alpha} \mathcal{B}
 \end{pmatrix}
: D \times D \subset H \times H \to H \times H \,.
\end{align*}
Then $\sigma(\mathcal{T}) \cap \dot{\imath}\setR \subset \set{0}$.
\end{lem}
\begin{proof}
 We divide the proof into several steps and start with
 \\
 Claim 1: $\mathcal{T}$ is closed.
 \\
 Proof of Claim 1: We have to show that
  $(u_n) \subset D \times D, \,\, u_n \to u , \,\, \mathcal{T} u_n \to f$ implies $u \in D \times D, \,\, \mathcal{T} u = f$.
 For $(u_n)$ and $f$ given as above, we find
 \begin{align*}
  \mathcal{A} u_n^1 + \alpha \mathcal{B}u_n^2 \to f^1
  \qquad \text{and} \qquad
  -\alpha \mathcal{A}u_n^1 + \mathcal{B}u_n^2 \to f^2 \,,
 \end{align*}
 hence
 \begin{align*}
  u_n^1 \to u^1, \quad (1+\alpha^2) \mathcal{A}u_n^1 \to f^1 - \alpha f^2
  \qquad \text{and} \qquad
  u_n^2 \to u^2, \quad (1+\alpha^2) \mathcal{B}u_n^2 \to \alpha f^1 + f^2 
  \, .
 \end{align*}
 Since $\mathcal{A}$ and $\mathcal{B}$ are closed, we obtain $u^1,u^2 \in D$ and 
 \begin{align*}
   (1+\alpha^2) \mathcal{A}u^1 = f^1 - \alpha f^2 , \quad 
   (1+\alpha^2) \mathcal{B}u^2 = \alpha f^1 + f^2 \, ,
 \end{align*}
 which rewritten means that
 $u \in D \times D$ and $\mathcal{T} u = f$. Claim 1 is proved.
 
 Let now $t \in \setR$ with $t \not =0$ be given. For the statement of the lemma, we have 
 to show that 
 for every $f=(f^1,f^2) \in H \times H$
 there exists a unique element $u = (u^1,u^2) \in D \times D$ such that
 \begin{align*}
  \dot{\imath} t u^1 + \mathcal{A}u^1 + \alpha \mathcal{B}u^2 = f^1
  \qquad \text{and} \qquad
  \dot{\imath} t u^2 - \alpha \mathcal{A}u^1 + \mathcal{B}u^2 = f^2 \,.
 \end{align*}
 This is equivalent to
 \begin{align*}
  \dot{\imath} t (u^1 - \alpha u^2) + (1+\alpha^2) \mathcal{A}u^1 = f^1 - \alpha f^2
  \qquad \text{and} \qquad
  \dot{\imath} t (\alpha u^1 +u^2) + (1+\alpha^2) \mathcal{B}u^2 = \alpha f^1 + f^2 \,.
 \end{align*}
 Multiplying both equations by $-\dot{\imath}$ leads to
 \begin{align*}
  \mathcal{T}_0 {u^1 \choose u^2} = {- \dot{\imath} f^1 + \dot{\imath} \alpha f^2 \choose -\dot{\imath} \alpha f^1
  - \dot{\imath} f^2} \, ,
 \end{align*}
 where the linear operator $\mathcal{T}_0 : D \times D \subset H \times H \to H \times H$ is defined
 by
 \begin{align*}
  \mathcal{T}_0 =
  \begin{pmatrix}
   t - \dot{\imath} (1 + \alpha ^2) \mathcal{A} & - \alpha t
   \\
   \alpha t & t - \dot{\imath}(1+\alpha^2) \mathcal{B}
  \end{pmatrix}
  \,.
 \end{align*}
 In particular, we have to show that $\mathcal{T}_0$ is bijective. 
 To see this, we first notice
 that $\mathcal{T}_0$ is closed, which can be proved as in Claim 1. Next, we determine the adjoint of
 $\mathcal{T}_0$:
 \\
 Claim 2: The adjoint $\mathcal{T}_0^*$ of $\mathcal{T}_0$ is given by
 \begin{align*}
  \mathcal{T}_0^* =
   \begin{pmatrix}
   t + \dot{\imath} (1 + \alpha ^2) \mathcal{A} & \alpha t
   \\
   - \alpha t & t + \dot{\imath}(1+\alpha^2) \mathcal{B}
  \end{pmatrix}
 \end{align*}
 with domain of definition $D(\mathcal{T}_0^*) = D \times D$.
 \\
 Proof of Claim 2: Recall that the set $D(\mathcal{T}_0^*)$ is defined by
 \begin{align*}
  D(\mathcal{T}_0^*) = \left\{ u \in H \times H \, | \,  
  v \mapsto (\mathcal{T}_0v,u)_H 
  \text{ is continuous on } D \times D
\text{ with respect to }\norm{\cdot}_{H}
\right\} \, ,
 \end{align*}
 and $\mathcal{T}_0^* u$ for $u \in D(\mathcal{T}_0^*)$ is the unique vector such that
  $(\mathcal{T}_0v,u)_H = (v, \mathcal{T}_0^* u)_H$ for all $v \in D \times D$.
 Let now $u \in D \times D$ be given. Since $\mathcal{A}$ and $\mathcal{B}$ are self-adjoint,
 we find for all elements $v \in D \times D$ the expression
 \begin{align*}
  (\mathcal{T}_0 v ,u)_H 
  =&(v^1, t u^1 + \dot{\imath} (1+ \alpha^2) \mathcal{A}u^1 + \alpha t u^2)_H
  + (v^2,- \alpha t u^1 + t u^2 + \dot{\imath} (1+\alpha^2) \mathcal{B}u^2)_H \, .
 \end{align*}
 This implies $u \in D(\mathcal{T}_0^*)$ and
\begin{align*}
 \mathcal{T}_0^* u =
  \begin{pmatrix}
   t + \dot{\imath} (1 + \alpha ^2) \mathcal{A} & \alpha t
   \\
   - \alpha t & t + \dot{\imath}(1+\alpha^2) \mathcal{B}
  \end{pmatrix}
  { u^1 \choose u^2 } \, .
\end{align*}
Let now $u \in D(\mathcal{T}_0^*)$ be given. In particular, the mappings
\begin{align*}
 v^1 &\mapsto t(v^1,u^1)_H - \dot{\imath} (1+\alpha^2) ( \mathcal{A} v^1 ,u^1 )_H + \alpha t (v^1,u^2)_H
 \\
 v^2 &\mapsto -\alpha t (v^2,u^1)_H + t(v^2,u^2)_H - \dot{\imath} (1+\alpha^2) (\mathcal{B} v^2,u^2)_H
\end{align*}
are continuous with respect to $\norm{\cdot}_H$. We obtain
 $u^1 \in D(\mathcal{A}^*)= D(\mathcal{A}) = D$ and  $u^2  \in D(\mathcal{B}^*) = D(\mathcal{B}) = D$,
hence $u \in D \times D$. Claim 2 is proved.

We now show that $\mathcal{T}_0$ and $\mathcal{T}_0^*$ are injective:
\\
Claim 3: For all $u \in D \times D$ we have
 $\text{Re} (\mathcal{T}_0 u,u)_H = t \, \norm{u}^2_H$ and $\text{Re} (\mathcal{T}_0^* u,u)_H = t \, \norm{u}^2_H$.
\\
 Proof of Claim 3: For $u \in D \times D$ we find
\begin{align*}
 (\mathcal{T}_0u,u)_H 
 =& t \, \norm{u}^2_H + \alpha t \big( (u^1,u^2)_H - \overline{(u^1,u^2)_H} \, \big)
 - \dot{\imath} (1+ \alpha^2)(\mathcal{A}u^1,u^1)_H - \dot{\imath} (1+ \alpha^2) (\mathcal{B}u^2,u^2)_H \,.
\end{align*}
Since $\mathcal{A}$ and $\mathcal{B}$ are self-adjoint, we obtain $\text{Re} (\mathcal{T}_0 u,u)_H = t \, \norm{u}^2_H$. The same argument gives the result for $\mathcal{T}_0^*$. Claim 3 is proved.
\\
Claim 4: $R(\mathcal{T}_0)$ is a closed subspace of $H \times H$.
\\
Proof of Claim 4: Let $(f_n) = (\mathcal{T}_0 u_n) \subset R(\mathcal{T}_0)$ be a sequence with 
$f_n \to f$. From Claim 3 it follows
\begin{align*}
 \abs{t} \, \norm{u_n - u_m}^2_H &= \abs{\text{Re} (\mathcal{T}_0 u_n -\mathcal{T}_0 u_m, u_n-u_m)_H}
 \le \norm{f_n - f_m}_H \, \norm{u_n - u_m}_H \, ,
\end{align*}
thus $(u_n)$ is a Cauchy sequence and $u_n \to u$ for some $u \in H \times H$. Since $\mathcal{T}_0$ is closed, we obtain $u \in D \times D$ and $\mathcal{T}_0 u=f$. Claim 4 is proved.

Summarizing, we know that $\mathcal{T}_0$ is densely defined, closed, injective 
with closed range, and $\mathcal{T}_0^*$ is injective. With the help of the closed range 
theorem, we find
 $R(\mathcal{T}_0) = N(\mathcal{T}_0^*)^\perp = \set{0}^\perp =  H\times H$.
In particular, $\mathcal{T}_0$ is a bijection and the lemma is proved.
\end{proof}
A combination of Lemmas \ref{lem:L1 is sectorial}, \ref{lem:L2 is sectorial}, \ref{lem:analytic semigroup of L0}, and \ref{lem:Spectrum of T} yields:
\begin{corollary}
 \label{lem:analytic semigroup of L0 continued}
 Let $\epsilon >0$ be small enough. Then the spectrum of $\mathcal{L}_0^\epsilon$ satisfies 
  $\sigma(\mathcal{L}_0^\epsilon) \cap \dot{\imath}\setR =\set{0}$. Moreover, the linear mapping
  $e^{t \mathcal{L}_0^\epsilon} - I : R(\mathcal{L}_0^\epsilon) \to R(\mathcal{L}_0^\epsilon)$
 is a bijection for every $t>0$.
\end{corollary}
\section{The perturbation argument}
\label{sec:Nperturbation}
As already announced, we introduce (compared to \cite{alex_small}) an additional parameter in our evolution equation and replace the external magnetic field $h_\text{ext}$ by $\lambda h + \gamma$. We assume that
\begin{align*}
 h \in C^{0,\beta}(\setR,L^2) + C^{0,\beta}(\setR,L^\infty)\qquad (0<\beta<1)
\end{align*}
is a (fixed) $T$-periodic function and that $\lambda$, $\gamma$ are real parameters. If for example the function $h$ belongs to $C^{0,\beta}(\setR)$ and
  $\int_0^T h(t) \, dt = 0$,
 then $\gamma$ represents the time average of the external magnetic field. 
Since $\theta_\epsilon \not\in H^2$, we decompose the angle $\theta$ by
 $\theta = \theta_\epsilon + \vartheta$
with $\vartheta(t,\cdot) \in H^2$. Then $(LLG)_\epsilon$ reads as
\begin{align*}
 \begin{array}{ll}
  \varphi_t = R^\epsilon_1(t,\varphi,\theta_\epsilon + \vartheta,\gamma,\lambda) \, , & 
  \quad \lim_{x \to \pm \infty} \varphi(\cdot,x)= 0 \, ,
 \\
 \vartheta_t = R^\epsilon_2(t,\varphi,\theta_\epsilon + \vartheta,\gamma,\lambda) \, , & \quad \lim_{x \to \pm \infty} \vartheta(\cdot,x)= 0 \, ,
 \end{array}
\end{align*}
for $(\varphi,\vartheta)$. For $0<\beta<1$, we define the function spaces
\begin{align*}
 X = C^1([0,T],L^2) \cap C([0,T],H^2) \cap C^{0,\beta}_\beta(]0,T],H^2) \cap
 C^{1,\beta}_{\beta}(]0,T],L^2)
\end{align*}
and
 $Y = C([0,T],L^2) \cap C^{0,\beta}_\beta(]0,T],L^2)$.
Similarly to \cite[Lemma 3.3]{alex_small}, we find the following statement:
\begin{lem}
\label{lem:Nexistence}
 Let $\theta_\epsilon$ be the phase function of a rescaled N{\'e}el wall. Then there exist an open neighborhood $U_\epsilon$ of $(0,0)$ in $H^2\times H^2$ and an open neighborhood $V_\epsilon$ of $(0,0)$ in $\setR \times \setR$ such that
\begin{align*}
  \begin{array}{ll}
  \varphi_t = R^\epsilon_1(t,\varphi,\theta_\epsilon + \vartheta,\gamma,\lambda) \, , & \quad
  \varphi(0,\cdot)={\varphi_0} \, ,
 \\
 \vartheta_t = R^\epsilon_2(t,\varphi,\theta_\epsilon + \vartheta,\gamma,\lambda) \, , & \quad \vartheta(0,\cdot)={\vartheta_0} \, ,
 \end{array}
\end{align*}
possesses a unique solution $(\varphi,\vartheta)=\big(\varphi(\cdot,{\varphi_0},{\vartheta_0},\gamma,\lambda),\vartheta(\cdot,{\varphi_0},{\vartheta_0},\gamma,\lambda)\big)$ in $X \times X$ close to $(0,0)$ for all $({\varphi_0},{\vartheta_0}) \in U_\epsilon$ and $(\gamma,\lambda) \in V_\epsilon$. Moreover, the mapping
\begin{align*}
 ({\varphi_0},{\vartheta_0},\gamma,\lambda) \mapsto \big(\varphi(\cdot,{\varphi_0},{\vartheta_0},\gamma,\lambda),\vartheta(\cdot,{\varphi_0},{\vartheta_0},\gamma,\lambda)\big)
\end{align*}
is smooth and
\begin{gather*}
 {D_{({\varphi_0},{\vartheta_0})} \varphi(T,0,0,0,0) \choose D_{({\varphi_0},{\vartheta_0})} \vartheta(T,0,0,0,0)}{h_1 \choose h_2} =
 e^{T \mathcal{L}_0^\epsilon} {h_1 \choose h_2} \, ,
 \\
 {D_\gamma \varphi (T,0,0,0,0) \choose D_\gamma \vartheta (T,0,0,0,0)} h_3= 
 \frac{h_3}{\epsilon} \int_0^T e^{(T-s) \mathcal{L}_0^\epsilon} {\alpha \cos \theta_\epsilon \choose \cos \theta_\epsilon} \, ds
\end{gather*}
for all $h_1,h_2 \in H^2$, $h_3 \in \setR$.
\end{lem}
\begin{proof}
 Thanks to the embedding $H^1 \hookrightarrow L^\infty$, we can choose an open neighborhood $X_0$ of $0$ in X such that
 $\sup_{0 \le t \le T} \norm{\varphi(t,\cdot)}_{L^\infty} \le \frac{\pi}{4}$
for all $\varphi \in X_0$. We now define $F:X_0\times X \times \setR \times \setR \to Y\times Y$ by
\begin{align*}
 F(\varphi,\vartheta,\gamma,\lambda) = \big( \varphi_t - R^\epsilon_1(\cdot,\varphi,\theta_\epsilon + \vartheta,\gamma,\lambda) , \vartheta_t - R^\epsilon_2(\cdot,\varphi,\theta_\epsilon + \vartheta,\gamma,\lambda) \big)
\end{align*}
and $G:X_0\times X \times H^2\times H^2 \times \setR \times \setR \to Y\times Y \times H^2 \times H^2$ by
\begin{align*}
 G(\varphi,\vartheta,{\varphi_0},{\vartheta_0},\gamma,\lambda) = \big( F(\varphi,\vartheta,\gamma,\lambda) , \varphi(0)-{\varphi_0},\vartheta(0)-{\vartheta_0}\big) \,.
\end{align*}
The embedding $H^1 \hookrightarrow L^\infty$, the choice of $X_0$, and Lemma \ref{lem:Nphasefunction} imply that $F$ and $G$ are well-defined and smooth. For example, we use the fact $\cos \theta_\epsilon \in L^2$ to see that
\begin{align*}
 \abs{\cos(\theta_\epsilon + \vartheta)} \le \abs{\cos \theta_\epsilon \cos \vartheta} + \abs{\sin \theta_\epsilon \sin \vartheta} 
\le \abs{\cos \theta_\epsilon} + \abs{\vartheta} \quad \in L^2 
\end{align*}
for every $\vartheta \in L^2$.
We already know that  $G(0,0,0,0,0,0) = (0,0,0,0)$
since the rescaled N{\'e}el wall is a stationary solution for LLG with $h_\text{ext}=0$. Moreover, the equation
\begin{align*}
 D_{(\varphi,\vartheta)}G(0,0,0,0,0,0) [\varphi,\vartheta] = (f_1,f_2,g_1,g_2)
\end{align*}
is equivalent to
\begin{align*}
 { \varphi_t \choose \vartheta_t } = \mathcal{L}^\epsilon_0 { \varphi(t) \choose \vartheta(t)} + { f_1(t) \choose f_2(t) } \, , \quad {\varphi(0) \choose \vartheta(0) } = { g_1 \choose g_2 } \,.
\end{align*}
Since $\mathcal{L}^\epsilon_0$ is sectorial, we obtain together with the optimal regularity result \cite[Corollary 4.3.6]{lunardi} that
\begin{align*}
 D_{(\varphi,\vartheta)}G(0,0,0,0,0,0): X \times X \to Y \times Y \times H^2 \times H^2
\end{align*}
is invertible. With the help of the implicit function theorem, we find open neighborhoods $U_\epsilon$ of $(0,0)$ in $H^2\times H^2$, $V_\epsilon$ of $(0,0)$ in $\setR \times \setR$, and a smooth mapping $(\varphi,\vartheta) : U_\epsilon \times V_\epsilon \to X_0 \times X$ such that
\begin{align*}
 \big(\varphi(\cdot,0,0,0,0),\vartheta(\cdot,0,0,0,0)\big) = (0,0)
\end{align*}
and
\begin{align*}
 G\big(\varphi(\cdot,{\varphi_0},{\vartheta_0},\gamma,\lambda),\vartheta(\cdot,{\varphi_0},{\vartheta_0},\gamma,\lambda),{\varphi_0},{\vartheta_0},\gamma,\lambda\big) = (0,0,0,0)
\end{align*}
for all $({\varphi_0},{\vartheta_0}) \in U_\epsilon$ and $(\gamma,\lambda) \in V_\epsilon$.
This in particular implies that 
\begin{align*}
 \big(\varphi(\cdot,{\varphi_0},{\vartheta_0},\gamma,\lambda),\vartheta(\cdot,{\varphi_0},{\vartheta_0},\gamma,\lambda)\big) \quad \in X \times X
\end{align*}
is the desired solution.

It remains to calculate the derivatives. With the help of the chain rule, we find
\begin{align*}
 (0,0,0,0) =& D_\varphi G(0) \circ D_{({\varphi_0},{\vartheta_0})} \varphi(0) [h_1,h_2] + D_\vartheta G(0) \circ D_{({\varphi_0},{\vartheta_0})} \vartheta(0) [h_1,h_2]  
+ D_{\varphi_0} G(0) [h_1] 
\\
&+ D_{\vartheta_0} G(0) [h_2]
\end{align*}
for all $h_1,h_2 \in H^2$. In particular, the function $v$ defined by
\begin{align*}
v = { D_{({\varphi_0},{\vartheta_0})}\varphi(\cdot,0,0,0,0) \choose D_{({\varphi_0},{\vartheta_0})}\vartheta(\cdot,0,0,0,0)}{h_1 \choose h_2}
\end{align*}
is a solution of
 $v_t = \mathcal{L}^\epsilon_0 v$ and $v(0) = (h_1, h_2)^T$,
hence $v(t) = e^{t \mathcal{L}^\epsilon_0} {h_1 \choose h_2}$ for all $0\le t \le T$. Similarly, we see that
\begin{align*}
 w = {D_\gamma \varphi (\cdot,0,0,0,0) \choose D_\gamma \vartheta (\cdot,0,0,0,0)} h_3
\end{align*}
is a solution of
\begin{align*}
 w_t = \mathcal{L}^\epsilon_0 w + \frac{h_3}{\epsilon}{\alpha \cos \theta_\epsilon \choose \cos \theta_\epsilon } \qquad \text{and} \qquad w(0)={ 0 \choose 0} \, .
\end{align*}
The variation of constants formula yields
\begin{align*}
 w(t) = \frac{h_3}{\epsilon} \int_0^t e^{(t-s) \mathcal{L}^\epsilon_0} { \alpha \, \cos \theta_\epsilon \choose \cos \theta_\epsilon } \, ds
\end{align*}
for all $0\le t \le T$.
The lemma is proved.
\end{proof}
For $(\varphi,\vartheta) = \big(\varphi(\cdot,{\varphi_0},{\vartheta_0},\gamma,\lambda),\vartheta(\cdot,{\varphi_0},{\vartheta_0},\gamma,\lambda)\big)$, we make use of the following equivalence:
{\itshape{
\begin{center}
 $(\varphi,\vartheta)$ defines a $T$-periodic solution for $(LLG)_\epsilon$ with $h_\text{ext} = \lambda h + \gamma$
\\
 if and only if $\big(\varphi(T),\vartheta(T)\big) = ({\varphi_0},{\vartheta_0})$.
\end{center}
}}
\noindent Because of that, we define the smooth function
\begin{align*}
 F_\epsilon: \, &U_\epsilon \times V_\epsilon \subset H^2 \times H^2 \times \setR \times \setR \to H^2 \times H^2 \times \setR  :
 \\
 &({\varphi_0},{\vartheta_0},\gamma,\lambda) \mapsto
 \big( \varphi(T,{\varphi_0},{\vartheta_0},\gamma,\lambda) - {\varphi_0}, \,\vartheta(T,{\varphi_0},{\vartheta_0},\gamma,\lambda) - {\vartheta_0}, {\vartheta_0}(0)\big)
\end{align*}
and remark that $F_\epsilon(0,0,0,0) = (0,0,0)$. To solve the equation $F_\epsilon({\varphi_0},{\vartheta_0},\gamma,\lambda)=(0,0,0)$
 for $\lambda \not=0$, we use the implicit function theorem and the statement of the next lemma.
\begin{lem}
 The linear operator $D_{({\varphi_0},{\vartheta_0},\gamma)}F_\epsilon(0,0,0,0): H^2 \times H^2 \times \setR \to H^2
 \times H^2 \times \setR$
 is invertible, provided $\epsilon>0$ is small enough.
\end{lem}
\begin{proof}
 With the help of Lemma \ref{lem:Nexistence}, we obtain the identity
\begin{align*}
 D_{({\varphi_0},{\vartheta_0},\gamma)}F_\epsilon(0,0,0,0)[h_1,h_2,h_3]
 =&\bigg( e^{T \mathcal{L}_0^\epsilon} {h_1 \choose
 h_2} - {h_1 \choose h_2} + \frac{h_3}{\epsilon} \int_0^T e^{(T-s) \mathcal{L}_0^\epsilon} {\alpha
 \cos \theta_\epsilon \choose \cos \theta_\epsilon} \, ds \, , \, h_2(0) \bigg)
\end{align*}
for all $h_1,h_2 \in H^2$, $h_3 \in \setR$.
Moreover, Lemma \ref{lem:L0 splitting} induces the splitting
 \begin{align*}
  H^2 \times H^2 = X_1 \oplus X_2 =  N(\mathcal{L}_0^\epsilon) \oplus (H^2 \times H^2) \cap R(\mathcal{L}_0^\epsilon)
 \end{align*}
 with projections $P_1$ and $P_2$ defined by
 \begin{align*}
  P_1 &: H^2 \times H^2 \to X_1: u={u_1 \choose u_2} \mapsto 
 {0 \choose t(u) \theta'_\epsilon}\,,\qquad P_2=I-P_1 \,,
 \end{align*}
 where 
  $t(u) = (\alpha u_1 + u_2,\theta'_\epsilon)_{L^2}/\norm{\theta'_\epsilon}_{L^2}^2$.
 From Corollary \ref{lem:analytic semigroup of L0 continued} we know that $e^{T \mathcal{L}_0^\epsilon} -I : R(\mathcal{L}_0^\epsilon) \to R(\mathcal{L}_0^\epsilon)$ is invertible, and thanks to the smoothing property of $e^{T \mathcal{L}_0^\epsilon}$, that is  $e^{T \mathcal{L}_0^\epsilon}(L^2\times L^2) \subset H^2\times H^2$,
we obtain that $e^{T \mathcal{L}_0^\epsilon} -I : X_2 \to X_2$
is an isomorphism. Moreover, we have $e^{T \mathcal{L}_0^\epsilon} u = u$ for all  $u \in X_1$. For a given element $(f_1,f_2,r) \in H^2 \times H^2 \times \setR$, the equation $D_{({\varphi_0},{\vartheta_0},\gamma)}F_\epsilon(0,0,0,0)[h_1,h_2,h_3] = (f_1,f_2,r)$
 is equivalent to
\begin{align*}
 P_1 { f_1 \choose f_2 } &= \frac{h_3}{\epsilon} \int_0^T e^{(T-s) \mathcal{L}_0^\epsilon} P_1 {\alpha \cos \theta_\epsilon \choose \cos \theta_\epsilon } \, ds \, ,
\\
P_2 { f_1 \choose f_2 } &= \big( e^{T \mathcal{L}_0^\epsilon} - I \big) P_2 { h_1 \choose h_2}
+ \frac{h_3}{\epsilon} \int_0^T e^{(T-s) \mathcal{L}_0^\epsilon} P_2 {\alpha \cos \theta_\epsilon \choose
\cos \theta_\epsilon } \, ,
\\
h_2(0)&= r \,.
\end{align*}
The properties of $\theta_\epsilon$ (see Lemma \ref{lem:Nphasefunction}) imply that $t(\alpha \cos \theta_\epsilon, \cos \theta_\epsilon) >0$ and we find
\begin{align*}
 \frac{h_3}{\epsilon} \int_0^T e^{(T-s) \mathcal{L}_0^\epsilon} P_1 {\alpha \cos \theta_\epsilon \choose \cos \theta_\epsilon } \, ds = \frac{T h_3}{\epsilon} { 0 \choose t(\alpha \cos \theta_\epsilon, \cos \theta_\epsilon ) \theta'_\epsilon } \,.
\end{align*}
In particular, $h_3$ is uniquely determined by $P_1{f_1\choose f_2}$. We now also obtain a unique $P_2 { h_1 \choose h_2 }$. The requirement $h_2(0) = r$ fixes $P_1 {h_1 \choose h_2}$ since
\begin{align*}
 { h_1 \choose h_2 } = P_1 {h_1 \choose h_2} + P_2 {h_1 \choose h_2} = {0 \choose t \theta'_\epsilon}
 + P_2 {h_1 \choose h_2}
\end{align*}
and $\theta'_\epsilon(0) >0$. The lemma is proved.
\end{proof}
Finally, we can state the main result of this paper.
\begin{theorem}
\label{thm:3}
Let $h \in C^{0,\beta}(\setR,L^2) + C^{0,\beta}(\setR,L^\infty)$ $(0<\beta<1)$ be a $T$-periodic function and $\epsilon >0$ be small enough. Then there exist an open ball 
$B \subset \setR$ centered at $0$ and smooth functions ${\varphi_0},{\vartheta_0}: B \to H^2$, $\gamma : B \to \setR$ such that
\begin{align*}
  (\varphi,\theta) = \big(\varphi(\cdot,{\varphi_0}(\lambda),{\vartheta_0}(\lambda),\gamma(\lambda),\lambda), \theta_\epsilon + 
  \vartheta(\cdot,{\varphi_0}(\lambda),{\vartheta_0}(\lambda),\gamma(\lambda),\lambda) \big)
\end{align*}
is a $T$-periodic solution for $(LLG)_\epsilon$ with external magnetic field $h_\text{ext} = \lambda h+ \gamma(\lambda)$ for every $\lambda \in B$. In particular, $m=(\cos \varphi \cos \theta, \cos \varphi \sin \theta , \sin \varphi)$ is a $T$-periodic solution for the rescaled LLG with external magnetic field $h_\text{ext}$.
By scaling the statement carries over to the original LLG.
\end{theorem}
\noindent {\bf Remarks.} 
\begin{enumerate}
\item[(i)] {\itshape The ``correction term'' $\gamma(\lambda)$ in $h_\text{ext} = \lambda h + \gamma(\lambda)$ corresponds to a compatibility condition.}
\item[(ii)] {\itshape A similar result is true for Bloch walls in thick layers. We plan to present this in a future paper.}
\end{enumerate}

\medskip

\noindent {\bf Acknowledgments.}
This work is part of the author's PhD thesis prepared at the Max Planck Institute for Mathematics in the Sciences (MPIMiS) and submitted in June 2009 at the University of Leipzig, Germany. The author would like to thank his supervisor Stefan M{\"u}ller for the opportunity to work at MPIMiS and for having chosen an interesting problem to work on. Furthermore, the author would like to thank Helmut Abels for helpful discussions and hints on the subject. Financial support from the International Max 
Planck Research School `Mathematics in the Sciences' (IMPRS) is also acknowledged.

\bibliographystyle{abbrv}
\bibliography{article-NeelWall}

\bigskip\small

\noindent{\sc NWF I-Mathematik, Universit\"at Regensburg,  93040 Regensburg}\\
{\it E-mail address}: {\tt alexander2.huber@mathematik.uni-regensburg.de}

\end{document}